\documentclass[a4paper]{article}

\usepackage{microtype}

\usepackage[utf8]{inputenc}

\usepackage[margin=3.5cm]{geometry}
\usepackage{authblk}
\usepackage{amsmath, amsthm, amsfonts, amssymb}
\usepackage{appendix}
\usepackage{hyperref}
\usepackage{mathtools}
\usepackage{lipsum}
\usepackage{enumitem}
\usepackage{tikz}
\usepackage{tikz-cd}
\usetikzlibrary{arrows, calc, cd, fit}
\usepackage[bb=dsserif]{mathalpha}
\usepackage{alltt}
\usepackage{algpseudocode}
\usepackage{algorithm}
\usepackage{nicematrix}
\usepackage{tabularx}
\usepackage{subcaption}

\newcommand{\blank}{{-}}
\newcommand{\Hom}{\mathrm{Hom}}

\newcommand{\id}{\mathrm{id}}
\newcommand{\Dnk}[2]{\Delta_{\hspace{1pt}#1}^{\hspace{-2pt}(#2)}}
\DeclareMathOperator{\im}{im}

\DeclareMathOperator{\supp}{supp }
\DeclareMathOperator{\coker}{coker}
\DeclareMathOperator{\Proj}{proj}
\DeclareMathOperator{\Incl}{incl}
\DeclareMathOperator{\rank}{rank}
\DeclareMathOperator{\Span}{span}
\DeclareMathOperator{\St}{St}

\DeclareMathOperator{\stcplx}{SC}

\newcommand\quotient[2]{
        \mathchoice
            {
                \text{\raise1ex\hbox{$#1$}\Big/\lower1ex\hbox{$#2$}}%
            }
            {
                #1\,/\,#2
            }
            {
                #1\,/\,#2
            }
            {
                #1\,/\,#2
            }
    }
    
\newcommand{\setdef}[2]{
		\left\{ #1 \,\middle|\,#2 \right\}
}

\makeatletter
\setbox0\hbox{$\xdef\scriptratio{\strip@pt\dimexpr
    \numexpr(\sf@size*65536)/\f@size sp}$}
\newcommand{\scriptveryshortarrow}[1][4pt]{{%
    \vcenter{\hbox{\rule[\scriptratio\dimexpr-.2pt\relax]
              {\scriptratio\dimexpr#1\relax}{\scriptratio\dimexpr.4pt\relax}}}%
  \mkern-4mu\hbox{\let\f@size\sf@size\usefont{U}{lasy}{m}{n}\symbol{41}}}}
\newcommand{\smap}[3]{{#1}(#2\leq #3)}

\algnewcommand\algorithmicforeach{\textbf{for each}}
\algdef{S}[FOR]{ForEach}[1]{\algorithmicforeach\ #1\ \algorithmicdo} 

\theoremstyle{definition}

\newtheorem{theorem}{Theorem}
\newtheorem{proposition}[theorem]{Proposition}

\newtheorem{lemma}[theorem]{Lemma}
\newtheorem{corollary}[theorem]{Corollary}

\newtheorem*{remark}{Remark}
\newtheorem{definition}[theorem]{Definition}
\newtheorem{example}[theorem]{Example}

\title{Computing Minimal Injective Resolutions of Sheaves on Finite Posets }

\author{Adam Brown and Ond\v{r}ej Draganov}
\date{}

\begin{document}
\maketitle
\begin{abstract}
     In this paper we introduce two new methods for constructing injective resolutions of sheaves of finite-dimensional vector spaces on finite posets. Our main result is the existence and uniqueness of a minimal injective resolution of a given sheaf and an algorithm for its construction. For the constant sheaf on a simplicial complex, we give a topological interpretation of the multiplicities of indecomposable injective sheaves in the minimal injective resolution, and give asymptotically tight bounds on the complexity of computing the minimal injective resolution with our algorithm.  
\end{abstract}

\section{Introduction}

A common strategy for analyzing a complicated mathematical structure is to approximate or represent the given structure with a collection of simpler, or at least more familiar, objects; the goal is to reframe questions concerning the complex structure as questions about the building blocks which represent it. Illustrations of this strategy permeate mathematics. The focus of this paper is a particular instance of this phenomena: \emph{injective resolutions} of \emph{sheaves}.

Sheaves use algebra to model relationships between local and global properties of a topological space. When the topological space is a poset (with the Alexandrov topology), a sheaf, $F$, is defined by associating a finite-dimensional vector space, $F(\sigma)$, to each element, $\sigma$, and a linear map, $F(\sigma\le\tau):F(\sigma)\rightarrow F(\tau)$, to each relation $\sigma\le \tau$ (subject to commutativity requirements, see Definition \ref{def:sheaf}). The utility of this definition is also its foil: the high level of generality encompasses many pathologies. For example, each persistence module (including the multi-parameter ones) can be viewed as a sheaf on a poset, and all of the difficulties in analyzing multi-parameter persistence modules arise when studying sheaves. An \emph{injective resolution} represents a given sheaf (much like a barcode or persistence diagram represents a 1-dimensional persistence module) with an exact sequence of \emph{injective sheaves}, which admit many desirable properties. Efficient algorithms for computing injective resolutions are a first step toward applying well-established and powerful theoretical results from derived sheaf theory to persistent homology. In this paper we aim to present this theory in an explicit and computationally feasible framework. 

\paragraph*{Injective Resolutions.} An \emph{injective sheaf}  (Definition \ref{def:injective_sheaf_2}), $I$, is a sheaf such that each morphism of sheaves $G\rightarrow I$ (Definition \ref{def:natural-transformation}) can be extended to a morphism $F\rightarrow I$, whenever $G\subset F$. Injective sheaves admit many beneficial features which general sheaves lack (see, for example, Lemma \ref{lem:coker-injective}, Proposition \ref{prop:decomposition_of_injective_sheaves}, and Lemma \ref{lem:maps_between_injective_sheaves}). From the perspective of homological algebra, injective sheaves are the `basic' objects with which we aim to represent a general sheaf. However, standard operations in linear algebra are insufficient for such a representation. For example, if a sheaf is not already injective, then it does not decompose into a direct sum of injective sheaves. Instead, we will represent a given sheaf~$F$ with an injective resolution (Definition \ref{def:injective_resolution}): an exact sequence, $0\rightarrow F\rightarrow I^0\rightarrow I^1\rightarrow I^2\rightarrow  \cdots$,
such that each $I^j$ is an injective sheaf. 

Injective resolutions, a fundamental ingredient for homological algebra, are used to study sheaves from the ‘derived’ perspective, i.e.\ as objects in a derived category. 
These derived categories unify and generalize several forms of cohomology, such as simplicial cohomology, de Rham cohomology, intersection cohomology, etc. For example, simplicial cohomology (and level-set persistent cohomology) can easily be computed from an injective resolution of the constant sheaf (see Example \ref{ex:constant-sheaf} and Section \ref{sec:derived-functors}), illustrating that even an injective resolution of the constant sheaf contains subtle topological information. Several recent works point to the potential benefits of applying derived sheaf theory to the study of persistent homology~ \cite{BerkoukPetit2021b,BerkoukGinotOudot,BerkoukPetit2021a,Curry2014,Schapira,KS2021}. We approach this subject from a computational perspective in order to help bridge gaps between applied topology and derived sheaf theory. With this goal in mind, we aim to limit the mathematical prerequisites of our approach whenever possible (a choice which often comes at the cost of brevity).


\paragraph*{Main Results.}
In this paper we develop methods for computing injective resolutions of sheaves of finite-dimensional vector spaces on finite posets. Our main contributions are:
\begin{enumerate}
    \item We establish the existence and uniqueness of a minimal injective resolution of a given sheaf (Theorem \ref{thm:main-result} and Corollary \ref{cor:existence}).
    \item For the constant sheaf on a simplicial complex, we give a topological interpretation of the multiplicity of an indecomposable injective sheaf in the minimal injective resolution, in terms of compactly supported cohomology (Theorem \ref{thm:multiplicities}).
    \item We introduce a non-inductive definition of a (non-minimal) injective resolution of a given sheaf (Section~\ref{sec:algo_non-inductive}). 
    \item We introduce an inductive algorithm for computing the minimal injective resolution (Algorithm \ref{algo:injective_resolution_tail}), and prove correctness of the algorithm (Section~\ref{sec:algo_minimal}).
    \item For the constant sheaf on a simplicial complex, we give an asymptotically tight bound on the complexity of Algorithm \ref{algo:injective_resolution_tail} (Proposition \ref{prop:upper_bound_on_SC} and Corollary \ref{cor:complexity}). 
\end{enumerate}
 As an application of our results, in Section \ref{sec:derived-functors} we explicitly describe the right derived pushforwards, $R^\bullet f_\ast $ and $R^\bullet f_! $, and show that traditional (level-set) multi-parameter persistence modules can be recovered from $R^\bullet f_\ast k_\Sigma$.

\paragraph{Comparison to Prior Work.} Derived sheaf theory is a rich subject which has been thoroughly developed over several decades. There are multiple textbooks on sheaf theory \cite{Bredon1997,Iversen,KashiwaraSchapira1994}, and many publications which study sheaves on finite topological spaces. In \cite{Shepard1985}, Shepard relates sheaves on finite cell complexes (viewed as posets) to the classical setting of constructible sheaves on stratified topological spaces. In \cite{Ladkani2008}, Ladkani studies the homological properties of finite posets and introduces combinatorial criteria guaranteeing derived equivalences between categories of sheaves. In \cite{Curry2014}, Curry establishes the connection between sheaf theory and persistent homology. More recently, several publications expand on the work initiated by Curry on applications of derived sheaf theory to persistent homology~\cite{BerkoukPetit2021b,BerkoukGinotOudot,BerkoukPetit2021a,Schapira,KS2021}. 

Motivated by the above work, and by the potential to develop new techniques for computational topology, we aim to establish results on computational aspects of derived sheaf theory for finite topological spaces. The contributions of this paper are the first of our knowledge to describe and analyze efficient algorithms for computing injective resolutions of sheaves on finite posets. While this is a necessary first step toward utilizing the machinery of derived categories in computational topology, there is still more work to be done. In Section~\ref{sec:discussion}, we describe several future directions of research which stem from this paper. 

\begin{remark}
We should comment on a matter of perspective and terminology. Over finite posets, sheaves are closely related (and sometimes equivalent, as in the case of modules over the incidence algebra of a poset), to several other mathematical objects studied by various research communities. Specifically, a great deal of work has been done in commutative algebra on minimal projective and free resolutions of modules over various kinds of algebras. However, in the present paper we choose to focus on the perspective and terminology which most closely aligns with classical sheaf theory, in order to preserve intuition from that discipline. 
\end{remark}



\section{Background and Preliminary Results}

In this paper we study finite-dimensional vector space valued sheaves on finite posets. We begin by recalling preliminary definitions and results, drawing heavily from \cite{Curry2014,Shepard1985,Ladkani2008}. Throughout the paper we fix a field $k$. 

For $\pi,\tau$ in a poset $\Pi$, we write $\pi<_1\tau$ if $\pi\lneq\tau$, and there is no other element between $\pi$ and $\tau$.
We use some of the standard terminology from simplicial complexes for general posets: the \emph{star} of an element $\sigma$ is $\St\sigma=\setdef{\tau\in\Pi}{\sigma\leq\tau}$, the \emph{boundary} is $\textrm{bnd}(\sigma)=\setdef{\tau\in\Pi}{\tau<_1\sigma}$, and the \emph{coboundary} is $\textrm{cobnd}(\sigma)=\setdef{\tau\in\Pi}{\sigma<_1\tau}$.

\begin{definition}
\label{def:sheaf}

A sheaf $F$ on a finite poset $\Pi$ is an assignment of a finite-dimensional $k$-vector space $F(\pi)$ to each element $\pi\in\Pi$, and an assignment of a linear map
\[
F(\tau\le\gamma):F(\tau)\rightarrow F(\gamma),
\]
to each face relation $(\tau\le\gamma)\in\Pi$, such that 
\begin{enumerate}
    \item $F(\tau\le\tau)=\id_{F(\tau)}$
    \item $F(\tau\le\gamma)\circ F(\sigma\le\tau) = F(\sigma\le\gamma)$
\end{enumerate}
for each triple $\sigma\le\tau\le\gamma\in\Pi$. 
\end{definition}

\begin{example}\label{ex:constant-sheaf}
The \emph{constant sheaf}, denoted $k_\Pi$, on a poset $\Pi$, assigns to each element $\pi\in\Pi$ the one-dimensional vector space, $k$, and to each relation, $(\pi\le \tau)\in\Pi$, the identity map $\id_k$. 
\end{example}

\begin{definition}
\label{def:natural-transformation}
A natural transformation, $\eta:F\rightarrow G$, between two sheaves on $\Pi$, is a collection of linear maps 
$
\eta(\pi):F(\pi)\rightarrow G(\pi)
$
for each $\pi\in\Pi$, such that 
\[
G(\tau\le\gamma)\circ \eta(\tau) = \eta(\gamma)\circ F(\tau\le\gamma),
\]
for each $(\tau\le\gamma)\in \Pi$. For a natural transformation $\eta:F\rightarrow G$, the kernel, cokernel, image, and coimage are taken point-wise, defining sheaves on $\Pi$: 
\[(\ker\eta)(\pi):=\ker (\eta(\pi)),\qquad (\ker\eta)(\pi\le\tau):=F(\pi\le\tau)\vert_{\ker\eta(\pi)}.\]
Moreover, if $\ker\eta(\pi)=0$ for each $\pi\in\Pi$, we say that $\eta$ is \emph{injective}. We write $G/F \coloneqq\coker\eta$ if $\eta$ is an injection clear from the context. 

\end{definition}

\begin{definition}
\label{def:injective_sheaf_2}
    A sheaf $I$ is called injective if for each injective natural transformation $A\hookrightarrow B$, any given natural transformation $A\rightarrow I$ can be extended to $B\rightarrow I$:
    \begin{center}
    \begin{tikzcd}
    0 \ar{r} & A \ar{r} \ar{dr}[swap]{\forall} & B \ar[dashed]{d}{\exists} \\
    & & I
    \end{tikzcd}
    \end{center}
\end{definition}

This condition is always satisfied for sheaves over a single point space (i.e., the assignment of a point to a single vector space): we can extend any linear map on a subspace to the whole space by mapping a complement space to 0. This property does not hold in general for sheaves. The following two examples show sheaves that are not injective.

\begin{example}\label{ex:non-injective_sheaf}
We define a sheaf $F$ which does not satisfy the condition of Definition~\ref{def:injective_sheaf_2}. Let us fix a vector space $W$, and define two sheaves, $F$ and $G$, on a poset with two elements $\sigma\leq \tau$. Let $F_{\sigma} = 0$ and $F_{\tau} = W$, and $G_{\sigma} = W = G_{\tau}$ with $\smap G{\sigma}{\tau} = \id$. Then $F$ embeds into $G$, and we claim that $F\overset{\id}{\rightarrow} F$ can not be extended to $G\rightarrow F$.\\
    \begin{minipage}{\textwidth}
    \centering
    \vspace{10pt}
    \begin{tikzcd}
         & & \\
        0 \ar{r} & F \ar{r} \ar{dr}[swap]{\id} & G \ar[dashed]{d}{?} \\
        & & F
    \end{tikzcd}%
    \hspace{40pt}
    \begin{tikzcd}
        &[-18pt] F & G & F \\
        \tau & W \ar{r}{\id} \ar[bend left]{rr}{\id} & W \ar{r}{?} & W \\
        \sigma & 0 \ar{u}{0} \ar{r}{0} \ar[bend right]{rr}{0} & W \ar{u}{\id} \ar{r}{?} & 0 \ar{u}{0}
    \end{tikzcd}
    \vspace{10pt}
    \end{minipage}
Indeed, the only way to make the right square commute is for both maps to be $0$, but then the top triangle does not commute.
\end{example}
We can use the same reasoning for a sheaf on any poset with a non-zero vector space one step above a zero vector space. Below we demonstrate one other obstruction to injectivity.

\begin{example}\label{ex:non-injective_sheaf_2}
We consider a three-element ``V'' shaped poset, a vector space $W$, and two different endomorphisms $f,g: W\rightarrow W$. We define sheaves $A$, $B$ and $F$ as follows:

    \begin{minipage}{.95\textwidth}
    \centering
    \vspace{10pt}
    \begin{tikzcd}[column sep = 7pt]
        & A & \\[-10pt]
        W & & W \\
        & 0 \ar{lu} \ar{ru} &
    \end{tikzcd}%
    \hspace{30pt}
    \begin{tikzcd}[column sep = 7pt]
        & B & \\[-10pt]
        W & & W \\
        & W \ar{lu}{f} \ar[swap]{ru}{g} &
    \end{tikzcd}%
    \hspace{30pt}
    \begin{tikzcd}[column sep = 7pt]
        & F & \\[-10pt]
        W & & W \\
        & W \ar{lu}{\id} \ar[swap]{ru}{\id} &
    \end{tikzcd}%
    \vspace{10pt}
    \end{minipage}
    
    We claim that $F$ does not satisfy the condition in Definition~\ref{def:injective_sheaf_2}. The sheaf $A$ embeds into $B$, and we choose $\alpha: A\rightarrow F$ to be the analogous embedding. To define an extension $\beta: B\rightarrow F$, we only have a choice for the bottom map $\beta_0: W\rightarrow W$. However, commutativity requires $f = \beta_0 = g$, which is impossible to satisfy, since $f\neq g$.
\end{example}

Avoiding the obstructions above, we define the simplest injective sheaves as follows.

\begin{definition}[{cf. \cite[Definition 7.1.3]{Curry2014}}] 
\label{def:indecomposable_injective_sheaf}
    For each $\pi\in\Pi$, we define an indecomposable injective sheaf $[\pi]$ as
    
    \begin{minipage}{.4\textwidth}
        \begin{align*}
            [\pi](\gamma) \coloneqq
            \begin{cases}
                k &\text{ if $\gamma\leq \pi$,} \\
                0 &\text{ otherwise,}
            \end{cases}
        \end{align*}
    \end{minipage}
    \begin{minipage}{.1\textwidth}
        \vspace{.4cm}
        with
    \end{minipage}
    \begin{minipage}{.4\textwidth}
        \begin{align*}
            [\pi](\gamma\le \tau) \coloneqq
            \begin{cases}
                \id &\text{ if $\gamma\le\tau\leq \pi$,} \\
                0 &\text{ otherwise.}
            \end{cases}
        \end{align*}
    \end{minipage}
  
\end{definition}
   For $n\in\mathbb{Z}_{\ge 0}$, we denote by $[\pi]^n$, the direct sum $\bigoplus_{i=1}^n[\pi]$. For a vector space $V$, we denote by $[\pi]^V$, the sheaf $[\pi]^{\dim V}$, with an implicitly fixed isomorphism between $k^{\dim V}$ and $V$.
   
The following results can be found in \cite{Curry2014} and \cite{Shepard1985} for sheaves on cell complexes. We give a straightforward generalization of the results to sheaves on any finite poset. 

\begin{lemma}[{cf. \cite[Lemma 7.1.5]{Curry2014}}]\label{lem:elementary_injective_sheaf}
    Indecomposable injective sheaves are injective.
\end{lemma}
\begin{proof}
    We show that $I=[\pi]$ for a fixed poset $\Pi$ and $\pi\in \Pi$ satisfies Definition~\ref{def:injective_sheaf_2}. Given an inclusion $A\xhookrightarrow{f} B$ and a natural transformation $\alpha: A\rightarrow I$, we need to find an extension $\beta: B\rightarrow I$. For the linear map $A(\pi)\xhookrightarrow{f(\pi)} B(\pi)$, there is a projection $A(\pi)\xleftarrow{g}B(\pi)$ such that $g f(\pi)=\id_{A(\pi)}$.
    We define
    \begin{align*}
        \beta(\gamma) \coloneqq
        \begin{cases}
            \alpha(\pi)\circ g\circ B(\gamma\leq \pi) &\text{ if $\gamma\leq \pi$,} \\
            0 &\text{ otherwise.}
        \end{cases}
    \end{align*}
    For every $\gamma$, this satisfies $\beta(\gamma) f(\gamma)=\alpha(\gamma)$, because if $\gamma\leq\pi$, then\[
    \beta(\gamma) f(\gamma) = \alpha(\pi)\, g\, B(\gamma\leq \pi) f(\gamma) = \alpha(\pi)\, g\, f(\pi) A(\gamma\leq \pi) = \alpha(\pi) A(\gamma\leq \pi) = \alpha(\gamma),
    \]
    and otherwise both sides are 0.
    For the commutativity conditions, consider $\gamma\leq\tau\leq\pi$. Then\[
    \beta(\tau) \smap B \gamma \tau = \alpha(\pi)\, g\, \smap B \tau \pi \smap B \gamma \tau = \alpha(\pi)\, g\, \smap B \gamma \pi = \beta(\gamma) = \smap I \gamma \tau \beta(\gamma).
    \]
    If $\gamma\leq\tau\not\leq\pi$, then both sides are $0$.
\end{proof}

\begin{lemma}[{cf.\ \cite[Lemma 1.3.1]{Shepard1985}}]
\label{lem:coker-injective}
    A direct sum of injective sheaves is injective. Additionally, if $I\xhookrightarrow{\alpha} J$ is an injective natural transformation with $I,J$ injective sheaves, then $J\cong I\oplus \coker\alpha  $, and $\coker\alpha$ is an injective sheaf. 
\end{lemma}
\begin{proof}
This proof is standard for any abelian category, we include a sketch for completeness. Suppose $I=A\oplus B$, with $A,B$ injective sheaves. Suppose $F\hookrightarrow G$ and $F\rightarrow I$. Then composition with projection gives maps $F\rightarrow A$ and $F\rightarrow B$. By injectivity of $A$ and $B$, each map extends to $G\rightarrow A$ and $G\rightarrow B$, respectively. The sum of these maps defines an extension $G\rightarrow I$, proving that $I$ is injective. The second claim follows by extending the identity map $I\rightarrow I$ to $J\rightarrow I$ by $\alpha$ and the injectivity of $J$. Then, the sum of the extension and the quotient map define an isomorphism $J\rightarrow I\oplus\coker\alpha$. The final claim follows by composing a given map $F\rightarrow\coker\alpha$ with the extension by zero map, $\coker\alpha\hookrightarrow J$, to get $F\rightarrow J$. Then, for $F\hookrightarrow G$, we define (by the injectivity of $J$) an extension $G\rightarrow J$. By post-composing with the projection map, we get the desired extension $G\rightarrow\coker\alpha$.
\end{proof}

\begin{proposition}[{cf. \cite[Lemma 7.1.6]{Curry2014}, \cite[Theorem 1.3.2]{Shepard1985}}] 
\label{prop:decomposition_of_injective_sheaves}
  Every injective sheaf is isomorphic to a direct sum of indecomposable injective sheaves.
\end{proposition}
\begin{proof}
We adapt the proof of \cite[Theorem 1.3.2]{Shepard1985} to the setting of finite posets on $n$ elements (rather than cell complexes). We fix some linear extension of the partial order, $(\pi_1,\dots,\pi_n)$, and let $\Pi_i=\left\{\pi_j\,\middle|\, j\leq i\right\}$. We will proceed with the proof by working inductively through this filtration of $\Pi$. We define support of a sheaf $I$ as\[
    \supp I \coloneqq \setdef{\pi\in\Pi}{I(\pi)\neq 0}.
\]

Assume that the result holds for injective sheaves supported on $\Pi_{i-1}$. Suppose $I$ is an injective sheaf with support contained in $\Pi_{i}$. If $\supp I\subseteq \Pi_{i-1}$, then the inductive assumption implies the result. Therefore, we are left to prove the result for $I$ such that $I(\pi_i)\neq 0$. Set $F_{\pi_i}$ to be the functor which assigns $I(\pi_i)$ to $\pi_i$ and the zero vector space to each other poset element (and the zero linear map to each poset relation). Then the identity map induces injective natural transformations
\[
F_{\pi_i}\xhookrightarrow{\alpha} I\qquad\text{and}\qquad F_{\pi_i}\hookrightarrow \bigoplus_{v\in B}[\pi_i],
\]
where $B$ is some basis of $I(\pi_i)$. Because $I$ is injective, we can extend $\alpha$ to a natural transformation
$
 \beta:  \bigoplus_{v\in B}[\pi_i]\rightarrow I.
$
It is injective, because for every $\sigma\leq\pi_i$, the linear map $I(\sigma\leq\pi_i)\beta(\sigma)=\beta(\pi_i)=\alpha(\pi_i)$ is injective.
By Lemma~\ref{lem:coker-injective}, this implies that 
\[
I\cong \coker \beta\oplus\bigoplus_{v\in B}[\pi_i],
\]
and that $\coker \beta$ is injective. Because $\supp\coker \beta\subseteq \Pi_{i-1}
$, the inductive hypothesis completes the proof.
\end{proof}

\subsection{Natural transformations between injective sheaves}\label{sec:maps}

Before we introduce injective resolutions, we take a brief detour to discuss natural transformations between injective sheaves.  To describe a map between two sheaves on a poset, in general, we need to give a linear map for each $\pi\in\Pi$. For two injective sheaves the situation is simpler.

Given a natural transformation $\varphi: I\rightarrow J$ between two injective sheaves, and a decomposition into indecomposable injective sheaves as in Lemma~\ref{prop:decomposition_of_injective_sheaves},
\[
I = \bigoplus_{i=1}^m [\pi_i]^{p_i}\hspace{15pt}\text{ and }\hspace{15pt}
J = \bigoplus_{j=1}^n [\sigma_j]^{s_j},
\]
$\varphi$ can be uniquely described by a collection of linear maps $f_{ij}: k^{p_i}\rightarrow k^{s_j}$, for each pair $i,j$ such that $\sigma_j\le\pi_i$. Moreover, each collection of linear maps defines a natural transformation. 

\begin{lemma}\label{lem:maps_between_injective_sheaves}
Suppose $I=\bigoplus_{i=1}^m [\pi_i]^{p_i}$ and $J= \bigoplus_{j=1}^n [\sigma_j]^{s_j}$. Then
\[\Hom(I,J) \ \ \cong\ \  \bigoplus_{\mathclap{\substack{i,j:\\ \sigma_j\le\pi_i}}} \Hom(k^{p_i},k^{s_j})\cong \bigoplus_{\mathclap{\substack{i,j:\\ \sigma_j\le\pi_i}}}k^{p_i s_j}, \]
where $\Hom(I,J)$ denotes the set of natural transformations from $I$ to $J$ and $\Hom(k^{p_i},k^{s_j})$ denotes the set of linear transformations from $k^{p_i}$ to $k^{s_j}$. 
\end{lemma}
\begin{proof}
Using projection and inclusion maps of the direct sum, $\Proj$, $\Incl$, a map $\varphi: I\rightarrow J$ can be decomposed as a sum of maps $\varphi_{ij}=\Proj_{[\sigma_j]^{s_j}}\circ \varphi \circ \Incl_{[\pi_i]^{p_i}}$ between the powers of indecomposable injective sheaves. Consider $\tau\in\Pi$. If $\tau\not\leq\sigma_j$, then $\varphi_{ij}(\tau)=0$.
Otherwise, $
\varphi_{ij}(\tau)
= \varphi_{ij}(\tau) \circ [\sigma_j]^{s_j}(\tau\leq\sigma_j)
= \varphi_{ij}(\sigma_j) \circ [\pi_i]^{p_i}(\tau\leq\sigma_j)
= \varphi_{ij}(\sigma_j)
$. This shows that $\varphi_{ij}$ is determined by the linear map $f_{ij}\coloneqq\varphi_{ij}(\sigma_j)$. Moreover, this map is necessarily $0$ whenever $\sigma_j\not\leq \pi_i$, and it can be any linear map otherwise.
\end{proof}

In other words, we can represent a natural transformation $\varphi:I\rightarrow J$ as just one $(\sum_{j=1}^n s_j) \times (\sum_{i=1}^m p_i)$ matrix with rows and columns labeled by the indecomposable injective sheaves in the decomposition of $J$ and $I$, respectively. The value $\varphi\big[[\sigma],[\pi]\big]$ at position labeled by $([\sigma],[\pi])$ gives the linear map between $[\pi](\sigma)$ and $[\sigma](\sigma)$. It is always $0$ if $\sigma\not\leq\pi$. Note that we can have more rows or columns labeled by a copy of the same $[\tau]$, and we do not distinguish different copies in our notation. The linear map $\varphi(\sigma)$ for $\sigma\in\Pi$ is described by the submatrix of all rows and columns labeled by elements in $\St\sigma$ (see Figure~\ref{fig:4-simplex_with_extra_edges_matrix}).
\section{Injective Resolutions}

In this section we give definitions of injective hull and resolution, and present our main theoretical results about the minimal injective resolutions.

\begin{definition}
An injective hull of a sheaf $F$ is an injective sheaf $I$ together with an injective natural transformation $F\hookrightarrow I$.
\end{definition}

\begin{definition}
A (bounded) complex of sheaves, denoted $A^\bullet$, is a sequence of sheaves $A^i$ and natural transformations $\mu^i$
\[
\cdots\rightarrow A^i\xrightarrow{\mu_i}A^{i+1}\xrightarrow{\mu_{i+1}}A^{i+2}\xrightarrow{\mu_{i+2}}\cdots
\]  
such that $\mu_{i+1}\circ\mu_i=0$ for each $i$, and $A^i=0$ for $|i|$ sufficiently large. A complex is exact if $\im\mu_i=\ker\mu_{i+1}$ for each $i$. A morphism $\alpha^\bullet:A^\bullet\rightarrow B^\bullet$ between complexes of sheaves is a collection of natural transformations $\alpha^i:A^i\rightarrow B^i$ such that the diagrams commute: 
\begin{center}
   \begin{tikzcd}
        A^i\ar{r}{\mu^i} \ar{d}{\alpha^i} & A^{i+1} \ar{d}{\alpha^{i+1}} \\
        B^i \ar{r}{\nu^i} & B^{i+1}
    \end{tikzcd}
\end{center}
\end{definition}
\begin{definition}
\label{def:injective_resolution}
An injective resolution of a sheaf $F$
is an exact sequence 
\[0\rightarrow F\xrightarrow{\alpha} I^{0}\xrightarrow{\eta^0} I^{1}\xrightarrow{\eta^{1}} I^{2} \xrightarrow{\eta^{2}} \dots\]
where $I^j$ is an injective sheaf for each $j$. We denote by $I^\bullet$ the complex 
\[
\cdots\rightarrow 0\rightarrow I^0\xrightarrow{\eta^0}I^1\xrightarrow{\eta^1}\cdots.
\]
\end{definition}
A classical result of sheaf theory is that each sheaf admits an injective resolution (though it need not be unique)~\cite{Iversen}. In the remainder of the paper, we will introduce and study explicit algorithms for computing injective resolutions of a given sheaf $F$.

\subsection{Minimal Injective Resolutions}
We will now define minimal injective resolutions, and show that they are unique up to isomorphism of complexes. We fix a sheaf $F$ on a finite poset $\Pi$. 

\begin{definition}\label{def:maximal-vectors}
    A vector $s\in F(\pi)$ is maximal if $F(\pi\le\tau)(s)=0$ for each $\tau\ge \pi$. Let $M_F(\pi)$ be the subspace of maximal vectors in $F(\pi)$, i.e. \[M_F(\pi)\coloneqq\bigcap_{\pi< \sigma}\ker F(\pi\leq\sigma).\]
    Note that it is sufficient to take only the intersection of $\ker F(\pi\le\sigma)$ for each $\pi<_1\sigma$. 
\end{definition}

\begin{definition}
    An injective resolution $I^\bullet$ of $F$ is minimal if, for each $i$, the number of indecomposable injective summands of $I^i$ is minimal among all injective resolutions of $F$. 
\end{definition}

\begin{theorem}\label{thm:main-result}
Let $I^\bullet$ be an injective resolution of a sheaf $F$. The following are equivalent: 
\begin{enumerate}
    \item $I^\bullet$ is minimal.
    \item For any injective resolution $J^\bullet$ of $F$, there exists a morphism of complexes $\delta^\bullet:I^\bullet\rightarrow J^\bullet $ such that $\delta^i$ is injective for each $i$. 
    \item For each $i>0$, each $\pi\in\Pi$, and each maximal vector $s\in I^i(\pi)$, $s\in\im \eta^i(\pi)$. For each $\pi\in\Pi$ and each maximal vector $s\in I^0(\pi)$, $s\in\im \alpha(\pi)$.
    \item For each $i$, each $\pi\in\Pi$, and each maximal vector $s\in I^i(\pi)$, $\eta^i(\pi)(s)=0$.
\end{enumerate}
\end{theorem}
\begin{proof}
\underline{$1\Rightarrow 4$}: Assume there exists a maximal vector $s\in I^i(\pi)$ such that $\eta^i(\pi)(s)\neq 0$. Let $[\pi]_s$ be an indecomposable injective subsheaf supported on the down-set of $\pi$ with $s\in[\pi]_s(\pi)$, and $ \hat{I}^i:= I^i/[\pi]_s$ the quotient, with quotient map $q_i$. Similarly, let $\hat{I}^{i+1}:= I^{i+1}/[\pi]_{\eta^i(\pi)(s)}$. By Lemma \ref{lem:coker-injective}, $I^i\cong \hat{I}^i\oplus [\pi]_s$, $I^{i+1}\cong \hat{I}^{i+1}\oplus [\pi]_{\eta^i(\pi)(s)}$, and $\hat{I}^i$, $\hat{I}^{i+1}$ are injective sheaves. Then 
\begin{center}
    \begin{tikzcd}
    &&0\ar{d}&0\ar{d}&&\\
    & 0 \ar{r} \ar{d}
    &{[\pi]_s}\ar{r}{\eta^i\vert_{[\pi]_s}}\ar{d}
    &{[\pi]_{\eta^i(\pi)(s)}}\ar{r}\ar{d}
    &0\ar{d}
    &\\
    \cdots \ar{r}
    &I^{i-1}\ar{r}{\eta^{i-1}}\ar{d} {\id}
    &I^{i}\ar{r}{\eta^i}\ar{d}{q_i}
    &I^{i+1}\ar{r}{\eta^{i+1}}\ar{d}{q_{i+1}}
    &I^{i+2}\ar{r}\ar{d}{\id}
    &{\cdots}\\
    \cdots\ar{r}
    & I^{i-1}\ar{r} {\hat\eta^{i-1}}\ar{d}
    & {\hat I^{i}}\ar{r}{\hat\eta^i} \ar{d}
    & {\hat I^{i+1}}\ar{r}{\hat\eta^{i+1}} \ar{d}
    & I^{i+2}\ar{r}\ar{d}
    &{\cdots}\\
    &0&0&0&0&
    \end{tikzcd}
\end{center}
with columns and the top two rows exact. We will show that the bottom row is also exact. Suppose $x\in\ker\hat\eta^{i-1}(\sigma)$ for some $\sigma\in\Pi$. Then $\eta^{i-1}(\sigma)(x)\in\ker q_i (\sigma) = [\pi]_s(\sigma)$. Because $I^\bullet$ is a chain complex, $\eta^i(\sigma)\circ\eta^{i-1}(\sigma)(x)=0$. By assumption the restriction of $\eta^i(\sigma)$ to $[\pi]_s(\sigma)$ is an injective linear map. Therefore, $\eta^{i-1}(\sigma)(x)=0$. This shows that $\ker\hat\eta^{i-1}=\ker\eta^{i-1}=\im\eta^{i-2}=\im\hat\eta^{i-2}$. By similar diagram chasing arguments, one can show that
\begin{align*}
    &\ker\hat\eta^i=q_i( \ker\eta^i) = q_i(\im\eta^{i-1})=\im\hat\eta^{i-1},\\
    &\ker\hat\eta^{i+1}=q_{i+1}(\ker\eta^{i+1})=q_{i+1}(\im\eta^i)=\im\hat\eta^i,\text{ and }\\
    &\im\hat\eta^{i+1}=\im\eta^{i+1}=\ker\eta^{i+2}=\ker\hat\eta^{i+2}.
\end{align*}   
Therefore, 
\[
0\rightarrow F \rightarrow I^0\rightarrow \cdots\rightarrow I^{i-1}\xrightarrow{\tilde\eta^{i-1}}\hat{I}_{i}\xrightarrow{\hat{\eta}_{i}}\hat{I}_{i+1}\xrightarrow{\hat\eta^{i+1}}I^{i+2}\rightarrow \cdots
\]
is an exact sequence, and an injective resolution of $F$, with fewer indecomposable injective summands than $I^\bullet$, which shows that $I^\bullet$ is not minimal.

\underline{$4\Leftrightarrow 3$}: Follows from the exactness of the injective resolution $I^\bullet$.

\underline{$3\Rightarrow 2$}: Let 
\[
0\rightarrow F\xrightarrow{\alpha} I^0\xrightarrow{\eta^0}I^1\xrightarrow{\eta^1}\cdots
\]
be an injective resolution of $F$ which satisfies criteria 3, and 
\[
0\rightarrow F\xrightarrow{\beta} J^0\xrightarrow{\lambda^0}J^1\xrightarrow{\lambda^1}\cdots
\]
be any injective resolution of $F$. We will inductively construct a chain complex morphism $\delta^\bullet:I^\bullet\rightarrow J^\bullet$ such that $\delta^i:I^i\rightarrow J^i$ is an injective natural transformation for each $i$. We begin by extending the natural transformation $\beta:F\rightarrow J^0$ through the injection $0\rightarrow F\xrightarrow{\alpha} I^0$, by the injectivity of $J^0$, resulting in the commutative diagram
\begin{center}
\begin{tikzcd}
0\ar{r}&F\ar{r}{\alpha}\ar{d}{\id}&I^0\ar{r}\ar[dashed]{d}{\delta^0}&\cdots\\
0\ar{r}&F \ar{r}{\beta} &J^0\ar{r}&\cdots
\end{tikzcd}
\end{center}
Similarly, we can extend the map $\alpha:F\rightarrow I^0$ to a map $\gamma^0:J^0\rightarrow I^0$, resulting in a commutative diagram 
\begin{center}
\begin{tikzcd}
0\ar{r}&F\ar{r}{\alpha}&I^0\ar{r}&\cdots\\
0\ar{r}&F \ar{u}{\id}\ar{r}{\beta} &J^0\ar{r}\ar[dashed]{u}{\gamma^0}&\cdots
\end{tikzcd}
\end{center}
We claim that $\ker\gamma^0\circ\delta^0=0$.
By assumption on $I^\bullet$, for each maximal vector $s\in I^0(\sigma)$, there exists $x\in F(\sigma)$ such that $\alpha(\sigma)(x)=s$. Because the above diagrams commute, $\beta(\sigma)(x)=\delta^0(\sigma)(s)$. Moreover, again by commutativity, 
\[\gamma^0(\sigma)\circ \delta^0(\sigma)(s)= \gamma^0(\sigma)\circ \beta(\sigma)(x)=\alpha(\sigma)(x)=s,\] which proves that $\ker \gamma^0\circ\delta^0=0$, because every non-zero vector maps to some non-zero multiple of a maximal vector via the sheaf maps. In particular, $\delta^0$ is injective.

We continue inductively. Suppose we have defined $\delta^0$ through $\delta^k$ and $\gamma^0$ through $\gamma^k$ such that  
\begin{center}
\begin{tikzcd}
0\ar{r}&F\ar{r}{\alpha}\ar{d}{\id}
&I^0\ar{r}\ar[swap, bend right = 30]{d}{\delta^0}
&\cdots \ar{r}{\eta^{k-1}}
&I^k\ar{r}\ar[swap, bend right = 30]{d}{\delta^k}
&I^{k+1}\ar{r}&\cdots\\
0\ar{r}&F \ar{r}{\beta} 
&J^0\ar{r}\ar[swap, bend right = 30]{u}{\gamma^0}
&\cdots\ar{r}{\lambda^{k-1}}
&J^k\ar{r}\ar[swap, bend right = 30]{u}{\gamma^k}
&J^{k+1}\ar{r}&\cdots
\end{tikzcd}
\end{center}
commutes for each square and $\ker\gamma^i\circ\delta^i=0$ for each $i$. 

Then $\delta^k$ and $\gamma^k$ induce natural transformations $\delta^k:\coker\eta^{k-1}\rightarrow \coker \lambda^{k-1}$ and $\gamma^k:\coker\lambda^{k-1}\rightarrow \coker \eta^{k-1}$, respectively. Using the injectivity of $I^{k+1}$ and $J^{k+1}$, we extend the maps from $\coker\eta^{k-1}\rightarrow J^{k+1}$ and $\coker\lambda^{k-1}\rightarrow I^{k+1}$, respectively:
\begin{center}
\begin{tikzcd}
I^k\ar{rr}{\eta^k}\ar[hookrightarrow, swap, bend right = 15]{ddd}{\delta^k}\ar[two heads]{dr}
& &I^{k+1}\ar[bend right = 15, swap, dashed]{ddd}{\delta^{k+1}}\\
 &\coker\eta^{k-1}\ar[ swap, bend right = 30]{d}{\delta^k}\ar[hookrightarrow]{ur}& \\
&\coker\lambda^{k-1}\ar[swap, bend right = 30]{u}{\gamma^k}\ar[hookrightarrow]{dr}&\\
J^k\ar{rr}{\lambda^k}\ar[two heads]{ur}\ar[swap, bend right = 15]{uuu}{\gamma^k}
&&J^{k+1}\ar[swap, bend right = 15, dashed]{uuu}{\gamma^{k+1}}
\end{tikzcd}
\end{center}
By diagram chasing, $\gamma^{k}\circ\delta^{k}(\im\eta^{k-1})\subset \im\eta^{k-1}$. By the inductive assumption, $\ker \gamma^{k}\circ\delta^{k}=0$. Therefore, as a map on $\coker\eta^{k-1}$, $ \gamma^{k}\circ\delta^{k}$ is injective. By an argument analogous to the above proof that $\ker \gamma^{0}\circ\delta^{0}=0$ is injective, we have that $\ker \gamma^{k+1}\circ\delta^{k+1}=0$. This implies that $\delta^{k+1}$ is injective. 
 

\underline{$2\Rightarrow 1$}: By condition 2, for each injective resolution $J^\bullet$, there are injective maps $\delta^i:I^i\hookrightarrow J^i$ for each $i$. The injectivity of $\delta^i$ implies that the number of indecomposable injective summands of $J^i$ is greater than that of $I^i$, which proves that $I^\bullet$ is minimal. 


\end{proof} 

The proof of the theorem yields several immediate corollaries.

\begin{corollary}\label{cor:existence}
For each sheaf $F$ on a finite poset $\Pi$, there exists a unique (up to isomorphism of chain complexes) minimal injective resolution. 
\end{corollary}
\begin{proof}
Notice that the proof of $1\Rightarrow 4$ in Theorem \ref{thm:main-result} shows that from any injective resolution $J^\bullet$ of $F$, and any maximal vector $s\in J^i(\pi)$ such that $\eta^i(\pi)(s)\neq 0$, we can construct an injective resolution $I^\bullet$ of $F$ by taking a quotient of $J^i$ and $J^{i+1}$ by $[\pi]_s$ and $[\pi]_{\eta^i(\pi)(s)}$, respectively. Therefore, by applying this procedure inductively, we can construct from any injective resolution $J^\bullet$, a minimal injective resolution $I^\bullet$. Existence of a minimal injective resolution then follows from the existence of injective resolutions. By Theorem \ref{thm:main-result} property 2, any two minimal injective resolutions must be isomorphic as chain complexes. 
\end{proof}

We also explicitly illustrate the existence of the minimal injective resolution in Section~\ref{sec:algo}, where we provide an algorithm to construct it.

\begin{corollary}\label{cor:length_of_resolution}
    The minimal injective resolution of a sheaf $F$ on a finite poset $\Pi$ of height~$d$ consists of at most $d+1$ non-zero injective sheaves.
\end{corollary}
\begin{proof}
    The length of the longest chain of non-zero vector spaces in $I^j$ is strictly decreasing in $j$ in the minimal injective resolution. This is implied by properties 3 and 4: If $I^j(\pi)=0$, then property 3 implies that there are no maximal vectors in $I^{j+1}(\pi)$. Therefore, if $I^j(\tau)=0$ for all $\tau\geq\pi$, then also $I^{j+1}(\tau)=0$ for all $\tau\geq\pi$. Moreover, if $I^j(\pi)\neq 0$ and $I^j(\tau)=0$ for all $\tau>\pi$, then all vectors in $I^j(\pi)$ are maximal, and by property 4 and the argument above, $I^{k}(\pi)=0$ for all $k>j$.
\end{proof}

\begin{corollary}\label{cor:minimal_injective_hull}
    If $F\xrightarrow{\alpha} I$ is an injective hull such that for all $\pi\in\Pi$, all maximal vectors of $I(\pi)$ are in $\im\alpha(\pi)$, then it is the minimal injective hull.
\end{corollary}
\begin{proof}
    The inductive construction in the proof of $3\Rightarrow 2$ in Theorem~\ref{thm:main-result} only depends on the initial segments of the resolution. Therefore, if the property $3$ is satisfied in an initial segment, then this initial segment injects in any injective resolution. In particular, this shows that if $F\xrightarrow{\alpha} I$ is an injective hull such that for all $\pi\in\Pi$ the maximal vectors in $I(\pi)$ are in $\im\alpha(\pi)$, then it is the minimal injective hull of $F$.
\end{proof}

\subsection{Indecomposable multiplicites of the minimal injective resolution}
Because the minimal injective resolution of a sheaf $F$ is unique, the multiplicity of an indecomposable injective sheaf in the minimal injective resolution is a well-defined invariant of $F$. It is natural to ask what topological information is captured with these multiplicities. Below we answer this question for the constant sheaf on a finite simplicial complex. 
\begin{definition}
  Let $I^\bullet$ be the minimal injective resolution of $F$. By $m^j_F(\sigma)$ we denote the multiplicity of $[\sigma]$ in $I^j$:
\[
I^j \cong \bigoplus_{\sigma\in\Pi}[\sigma]^{m^j_F(\sigma)}.
\]
Equivalently, we can define 
$m^j_F(\sigma):=\dim M_{I^j}(\sigma)$, 
where $M_{I^j}(\sigma)$ is as in Definition \ref{def:maximal-vectors}.
\end{definition}
\begin{theorem}\label{thm:multiplicities}
Let $\Sigma$ be a finite simplicial complex, $k_\Sigma$ the constant sheaf on $\Sigma$ (viewed as a poset with the face relation), and $H_c^\bullet (|\St\sigma|,k)$ be the singular cohomology with compact support of the geometric realization of $\St\sigma$. Then
\[
m^j_{k_\Sigma}(\sigma)= \dim H_c^{j+\dim\sigma}(|\St\sigma|,k).
\]
\end{theorem}
\begin{proof}
For $\sigma\in\Sigma$, there exists an abstract simplicial complex $K = \sigma^0\ast \left\{ \tau\setminus\sigma \,\middle|\, \tau\in\St\sigma\setminus\{\sigma\} \right\}$ with $\sigma^0\in K$ a zero simplex, such that $\St\sigma^0=\St\sigma$ as posets, and the geometric realization $|\St\sigma|$ is homeomorphic to $\mathbb{R}^{\dim\sigma}\times |\St\sigma^0|$. 
Let $I^\bullet$ be the minimal injective resolution of $k_\Sigma$. Then $I^\bullet\vert_{\St\sigma}$ is a minimal injective resolution of $k_{\St\sigma^0}$, because $\St\sigma=\St\sigma^0$ as posets. We will briefly abuse notation and think of $I^\bullet\vert_{\St\sigma}$ as a complex of injective sheaves on $\St\sigma^0$. 
Let $f:\St\sigma^0\rightarrow \text{pt}$. Because $\dim\sigma^0=0$, the map $f$ is a fibred cellular map in the sense of \cite[\S 3.3]{Shepard1985}. Therefore, by \cite[\S 1.6]{Shepard1985} (cf.\ \cite[Definition 5.1.15]{Curry2014}), $f_!(J)=\{s\in J(\sigma^0):s\text{ is a maximal vector}\}$, for each injective sheaf $J$ on $\St\sigma^0$. 
By standard results of sheaf theory (for example, \cite[Theorem 3.4.14]{Shepard1985} and \cite[Section 13.2]{Curry2014}), we have
\[R^jf_! I^\bullet\vert_{\St\sigma}\cong H^j_c(|\St\sigma^0|,k)
\]
(where $R^jf_!I^\bullet\vert_{\St\sigma}$ are the cohomology groups of the complex $f_!I^\bullet\vert_{\St\sigma}$; see \cite{Shepard1985,Bredon1997,Curry2014} for an introduction to the right derived functors $R^jf_!$). It follows from Theorem \ref{thm:main-result} that $ f_! I^\bullet\vert_{\St\sigma}$ is the complex
\[
0\rightarrow [\sigma^0]^{m_{k_\Sigma}^0(\sigma)}\rightarrow [\sigma^0]^{m_{k_\Sigma}^1(\sigma)}\rightarrow \cdots,
\]
with all chain maps equal to zero. Therefore, 
\[
\dim R^jf_! I^\bullet\vert_{\St\sigma}= m^j_{k_\Sigma}(\sigma).
\]
The result then follows from the isomorphisms
\[
 H_c^{j}(|\St\sigma^0|,k)\cong H_c^{j+\dim\sigma}(\mathbb{R}^{\dim\sigma}\times |\St\sigma^0|,k)\cong H_c^{j+\dim\sigma}( |\St\sigma|,k).
\]
\end{proof} 

\section{Algorithms for Computing Injective Resolutions}\label{sec:algo}

We describe two methods for constructing an injective resolution of a sheaf $F$ on a poset $\Pi$. 

\subsection{Injective Resolutions via the order complex}\label{sec:algo_non-inductive}

We begin with a non-inductive construction of a (not necessarily minimal) injective resolution of a given sheaf $F$. This section generalizes, from the constant sheaf to general sheaves, Lemma 1.3.17 of~\cite{Ladkani2008}. On a practical level, this allows one to compute $k$-th right derived functors without first computing the full injective resolution (see Section \ref{sec:derived-functors}).

\begin{definition}
    The order complex, $K(\Pi)$, of a finite poset $\Pi$, is the poset of strictly increasing chains $\pi_\bullet = \pi_0<\pi_1<\cdots<\pi_k$ in $\Pi$. The order complex has the structure of an abstract simplicial complex. Let $K^i(\Pi)$ denote the $i$-simplices of $K(\Pi)$, i.e. the set of chains $\pi_0<\pi_1<\cdots<\pi_i$ of length $i+1$.
\end{definition}
\begin{definition}
    A signed incidence relation on $K(\Pi)$ is an assignment to each pair of simplices $\sigma_\bullet,\gamma_\bullet\in K(\Pi)$ a number $[\sigma_\bullet:\gamma_\bullet]\in\{-1,0,1\}$, such that
    \begin{enumerate}
        \item if $[\sigma_\bullet:\gamma_\bullet]\neq 0$, then $\sigma_\bullet<_1\gamma_\bullet$, and 
        \item for each pair of simplices $(\sigma_\bullet,\gamma_\bullet)$,  \[\sum_{\tau_\bullet\in K(\Pi)}[\sigma_\bullet:\tau_\bullet][\tau_\bullet:\gamma_\bullet]=0.\] 
    \end{enumerate}
\end{definition}

\paragraph*{The construction.} Given a sheaf $F$ on $\Pi$, we define (recalling the notation of Definition \ref{def:indecomposable_injective_sheaf}) 
\[
I^k := \bigoplus_{\pi_\bullet\in K^k(\Pi)}[\pi_0]^{ F(\pi_k)}.
\]
Suppose $\pi_\bullet\in K^k(\Pi)$ and $\pi_\bullet <_1\tau_\bullet$ (i.e. the chain $\pi_\bullet$ is obtained from the chain $\tau_\bullet$ by removing one element). Then $\pi_0\ge \tau_0$ and $\pi_k\le\tau_{k+1}$. Therefore, 
\[
F(\pi_k\le\tau_{k+1})\in \Hom (F(\pi_k),F(\tau_{k+1}))\cong \Hom \left( [\pi_0]^{F(\pi_k)}, [\tau_0]^{ F(\tau_{k+1})}\right) . 
\]
Using this identification, we define the natural transformation
$
    \eta^k:I^k\rightarrow I^{k+1}
$
so that on the $\pi_\bullet$-summand $[\pi_0]^{F(\pi_k)}$ of $I^k$, 
\begin{align*}
    \eta^k\vert_{[\pi_0]^{ F(\pi_k)}}= \sum_{\pi_\bullet<_1\tau_\bullet}[\pi_\bullet:\tau_\bullet]F(\pi_k\le \tau_{k+1}),
\end{align*}
where $F(\pi_k\le \tau_{k+1})\in\Hom \left( [\pi_0]^{ F(\pi_k)}, [\tau_0]^{F(\tau_{k+1})}\right)$ is understood to have its codomain as the $\tau_\bullet$-summand of $I^{k+1}$. Let $\alpha:F\hookrightarrow I^0$ be the natural transformation given by the maps 
\[\alpha(\sigma)
\coloneqq\ 
\sum_{\sigma\le\gamma}F(\sigma\le\gamma)
\ :\ 
F(\sigma)\xhookrightarrow{\ \ \ \ }  \bigoplus_{\sigma\le\gamma}F(\gamma)
\ \reflectbox{ $\coloneqq$ } I^0(\sigma).\]

\begin{theorem}\label{thm:non-inductive}
The complex $0\rightarrow F\xrightarrow{\alpha} I^0\xrightarrow{\eta^0}I^1\xrightarrow{\eta^1} \cdots$ defined above is an injective resolution of $F$. 
\end{theorem}
\begin{proof}
By construction, each sheaf $I^j$ is injective, and each map $\eta^j$ (as well as $\alpha$) is a natural transformation. It remains to show that the sequence is an exact chain complex. It is enough to show that for each $\pi\in\Pi$, the sequence $0\rightarrow F(\pi)\xrightarrow{\alpha(\pi)}I^0(\pi)\xrightarrow{\eta^0(\pi)}I^1(\pi)\xrightarrow{\eta^1(\pi)}\cdots$ is exact. 

We first define a functor $T$ from the category of sheaves on $\Pi$ to the category of sheaves on $K(\Pi)$. To each sheaf $F$ on $\Pi$, let $T(F)$ be the sheaf on $K(\Pi)$ defined by associating to each chain $\tau_\bullet=\tau_0<\cdots<\tau_k$ the `terminal' vector space: $T(F)(\tau_\bullet) \coloneqq F(\tau_k)$ and $T(F)(\tau_\bullet\le\gamma_\bullet)=F(\tau_k\le \gamma_j)$ for $\tau_\bullet \in K^k(\Pi)$, $\gamma_\bullet\in K^j(\Pi)$. Because $T(\eta)(\tau_\bullet) \coloneqq \eta(\tau_k):T(F)(\tau_\bullet)\rightarrow T(G)(\tau_\bullet)$ for any natural transformation $\eta:F\rightarrow G$, it is clear that $T$ is an exact functor. 

Notice that $0\rightarrow F(\pi)\xrightarrow{\alpha(\pi)}I^0(\pi)\xrightarrow{\eta^0(\pi)}I^1(\pi)\xrightarrow{\eta^1(\pi)}\cdots$ is identical to the compactly supported cochain complex of the sheaf $T(F\vert_{\St\pi})$ on the simplicial complex $K(\St\pi)$ \cite[Definition 6.2.1 and Definition 6.2.3]{Curry2014}. Therefore, exactness in $I^0(\pi)$ follows from \[
    \ker\eta_0(\pi) 
    \cong \Gamma(T(F\vert_{\St\pi}))
    \cong F(\pi)
    \cong \im \alpha(\pi),
\] and it remains to prove a vanishing property for the cohomology of $T(F\vert_{\St\pi})$, namely that $H^j(K(\St\pi);T(F\vert_{\St\pi}))=0$ for $j>0$. 

Let $J^\bullet$ be an injective resolution of the sheaf $F\vert_{\St\pi}$ on the poset $\St\pi$. Because $T$ is an exact functor (and maps injective sheaves to injective sheaves), $T(J^\bullet)$ is an injective resolution of $T(F\vert_{\St\pi})$. Let $f:K(\St\pi)\rightarrow \text{pt}$. Then $R^jf_\ast(T(J^\bullet))\cong H^j(K(\St\pi);T(F\vert_{\St\pi}))$ (see Section \ref{sec:derived-functors}). Moreover, $R^jf_\ast (T(J^\bullet))$ is isomorphic to the $j$-th cohomology group of the complex of vector spaces $J^\bullet(\pi)$, which, by the exactness of $J^\bullet$, is zero for $j>0$. 
\end{proof} 

\subsection{Minimal injective resolutions via inductive algorithm} \label{sec:algo_minimal}

We first describe an explicit construction of the minimal injective hull of a sheaf $F$ on a poset $\Pi$, and then give an algorithm to inductively compute the minimal injective resolution with the minimal injective hull as the input. Lastly, we focus on constant sheaves, give an example, and analyze complexity of the algorithm.

\paragraph*{Minimal injective hull.}
To construct the minimal injective hull of $F$, we first find the space of maximal vectors $M_F(\pi)$ for each $\pi\in\Pi$ (see Definition \ref{def:maximal-vectors}). Then $M_F$, with zero linear maps, is a subsheaf of $F$. Recalling the notation described below Definition \ref{def:indecomposable_injective_sheaf}, we define\[
I^0 = \bigoplus_{\pi\in\Pi}[\pi]^{M_F(\pi)} \cong \bigoplus_{\pi\in\Pi}[\pi]^{\dim M_F(\pi)},
\]
where $[\pi]^{M_F(\pi)}$ is the injective sheaf with $[\pi]^{M_F(\pi)}(\sigma)=M_F(\pi)$ if $\sigma\leq\pi$, and $0$ otherwise. We can naturally include $M_F \xrightarrow{\gamma} I^0$, and extend this inclusion to $F\xrightarrow{\alpha} I^0$, using the injectivity of $I^0$. We choose the extension $\alpha=\sum_{\pi\in\Pi}\alpha_{\pi}$, where $\alpha_\pi$ is the extension of $\Proj_{[\pi]^{M_F(\pi)}}\!\circ\,\gamma$ to $F$ that we describe in the proof of Lemma~\ref{lem:elementary_injective_sheaf}. That is,
\begin{align*}
    \alpha_{\pi}(\sigma) \coloneqq
    \begin{cases}
        \Proj_{M_F(\pi)} \circ F(\sigma\leq \pi) &\text{ if $\sigma\leq \pi$,} \\
        0 &\text{ otherwise.}
    \end{cases}
\end{align*}

\begin{proposition}\label{prop:construction_of_minimal_injective_hull}
    This construction yields the minimal injective hull $F\xrightarrow{\alpha} I^0$.
\end{proposition}
\begin{proof}
We claim that $\alpha$ is injective. Let $u\in \ker\alpha(\sigma)$. Then\[
   0 = \alpha(\sigma)(u) = \sum_{\pi\in\Pi}\alpha_\pi(\sigma)(u) = \sum_{\sigma\leq\pi}\Proj_{M_F(\pi)} F(\sigma\leq \pi)(u),
\]
which is equivalent to $\Proj_{M_F(\pi)} F(\sigma\leq \pi)(u)=0$ for every $\pi\geq\sigma$, since the images of different $\alpha_{\pi}$ only intersect in $0$. But this means that $u=0$, because every non-zero vector is either maximal or maps onto some non-zero maximal vector via the sheaf maps.

By Corollary~\ref{cor:minimal_injective_hull}, the minimality of the injective hull is equivalent to the condition that every maximal vector of $I^0$ is in $\im\alpha$. This is satisfied, as $M_F(\pi)$ are exactly the maximal vectors in $I^0(\pi)$.
\end{proof}

We give an explicit formulation of an algorithm computing $\alpha(\pi)$ as Algorithm~\ref{algo:injective_hull_alpha}. We first fix bases in $F$. For each $\pi\in\Pi$, we fix a basis $B(\pi)=(v_1,\dots,v_l,w_{l+1},\dots,w_{l+k})$, with $l,k$ dependent on $\pi$, such that $(w_{l+1},\dots,w_{l+k})$ is a basis of $M_F(\pi)$, which we also use for $I^0(\pi)$. We assume that all maps $F(\pi\leq\sigma)$ are expressed with respect to those bases.

\begin{algorithm}[ht]
\caption{Minimal injective hull}\label{algo:injective_hull_alpha}
\hspace*{\algorithmicindent} \textbf{Input:} $F$ with fixed bases as described above, $\pi\in\Pi$ \\
\hspace*{\algorithmicindent} \textbf{Output:} $\alpha(\pi)$ as a $\left(\sum_{\sigma\leq\pi}\dim M_F(\sigma)\right)\times(\dim F(\pi))$ matrix
\begin{algorithmic}[1]
    \Procedure{Incl}{$\sigma$, $w\in M_F(\sigma)$}
        \State return inclusion of $w$ into $\bigoplus_{\pi<\tau}M_F(\tau)$ \Comment{just adding extra zeros}
    \EndProcedure
    \For{$v_i\in\{v_1,\dots,v_l\}$} \Comment{$(v_1,\dots,v_l,w_{l+1},\dots,w_{l+k})$ is the fixed basis of $F(\pi)$}
        \State $D\gets$ empty dictionary \Comment{keys: elemets $\pi\in\Pi$, values: vectors in $F(\pi)$}
        \State $D[\pi]\gets v_i$
        \State $u_i \gets 0$ \Comment{vector of length $\sum_{\pi<\tau}\dim M_F(\tau)$}
        \ForEach{$\sigma\geq\pi$ in some topological ordering}
            \If{$\sigma\in\textrm{Keys}(D)$ and $D(\sigma)\neq 0$}
                \State $w\gets D[\sigma]$
                \Comment{$D[\sigma] = F(\pi\leq\sigma)(v_i)$}
                \ForEach{$\tau>_1\sigma$}
                    \If{$\tau\not\in\mathrm{Keys}(D)$}
                        \State $D[\tau]\gets F(\sigma\leq\tau)(w)$
                        \State $u_i\gets u_i + \textsc{Incl}(\sigma,\, \Proj_{M_F(\tau)}(D[\tau]))$
                    \EndIf
                \EndFor
            \EndIf
            \State clear $D[\sigma]$ \Comment{optional, just to free up memory}
        \EndFor
        \State return a block matrix
            $\begin{pmatrix}%
                U & 0 \\%
                0 & I %
            \end{pmatrix}$,
            where $U=(u_1|\dots|u_l)$, \newline \hspace*{5.5cm} and $I$ is the identity matrix of order $k=\dim M_F(\pi)$
    \EndFor
\end{algorithmic}
\end{algorithm}

To express $\alpha(\pi)$ with respect to the fixed bases, we need to find the image of each $v_1,\dots,v_l, w_{1+l},\dots,w_{l+k}$. The maximal vectors $w_i$ are mapped identically to $M_F(\pi)\subseteq I_0(\pi)$. For the other vectors, $v_i$, we need to find $u_i\coloneqq \sum_{\pi<\tau}\Proj_{M_F(\tau)}F(\pi\leq\tau)(v_i)$. The algorithm does that while avoiding redundant computations. Each $\sigma$ is added to $D$ at most once. If it is added, then $D[\sigma]=F(\pi\leq\sigma)(v_i)$, and $\Proj_{M_F(\pi)}F(\pi\leq\sigma)(v_i)$ is added to $u_i$. If $\sigma$ is never added to $D$. then $F(\pi\leq\sigma)(v_i)=0$. In the end, $u_i$ contains the desired sum.

\paragraph*{Minimal injective resolution.}
Next, we describe an algorithm that takes $I^{k-1} \xrightarrow{\eta^{k-1}} I^k$ as the input, and gives $I^k \xrightarrow{\eta^k} I^{k+1}$ on the output. Essentially the same algorithm can be also used to compute $I^0 \xrightarrow{\eta^0} I^{1}$ from $F\xrightarrow{\alpha}I^0$, with the difference that $\alpha$ can not be stored the same way we store $\eta^k$---we discuss this in more detail later.

\paragraph*{Representing the sheaves and the maps.} We represent each sheaf $I^k$ as a tuple of poset elements, $(\pi_1,\dots,\pi_l)$, with possible repetitions, such that $I^k=\bigoplus_{i=1}^l[\pi_i]$. We refer to the elements in that tuple as \emph{generators} of $I^k$. We describe the natural transformation $\eta^k: I^k \rightarrow I^{k+1}$ by a matrix with columns labeled by the tuple of generators of $I^k$, and rows labeled by the tuple of generators of $I^{k+1}$, as discussed in Section~\ref{sec:maps}. In the following, we abuse the notation and denote matrices by the same symbols as the maps they represent. The linear map $\eta^k(\sigma)$ is described by a submatrix \[\eta^k(\sigma)\coloneqq\eta^k[\St\sigma,\St\sigma],\] where we take all the rows and columns that are labeled by simplices from $\St\sigma$. Note that Lemma~\ref{lem:maps_between_injective_sheaves} also implies that $\eta^k[\St\sigma,\Pi\setminus\St\sigma]$ is a zero matrix. Hence, $\eta^k[\St\sigma,\St\sigma]$ and $\eta^k[\St\sigma,\Pi]$ only differ by zero-columns.

\paragraph*{The algorithm going from $\xrightarrow{\eta^{k-1}}I^k$ to $\xrightarrow{\eta^{k}}I^{k+1}$.} The construction of $I^{k+1}$ and $\eta^k$ is described by Algorithm~\ref{algo:injective_resolution_tail}. We start with an empty sheaf and an empty matrix. Then we inductively add generators to $I^{k+1}$ and corresponding rows to $\eta^k$.
Fix a total order, $(\sigma_1,\dots,\sigma_n)$, extending the poset $\Pi$.
We go through the elements in reverse, and for each $\sigma$, we make sure that $\ker\eta^k(\sigma)= \im\eta^{k-1}(\sigma)$. We do that by adding new linearly independent rows to $\eta^k(\sigma)$ from $\big(\im\eta^{k-1}(\sigma)\big)^{\perp}=\left\{v\in k^{\dim I^{k}(\sigma)}\,\middle|\, \forall u\in\im\eta^{k-1}(\sigma): v\cdot u=0\right\}$. The entries of the new rows in positions outside of $\St\sigma$ are $0$. The new rows are labeled by $\sigma$, and for each added row we put a new element $\sigma$ in the tuple representing $I^{k+1}$.

\begin{algorithm}[h]
\caption{Step in the minimal injective resolution}\label{algo:injective_resolution_tail}
\hspace*{\algorithmicindent} \textbf{Input:} $I^k$, $\eta^{k-1}$ \\
\hspace*{\algorithmicindent} \textbf{Output:} $I^{k+1}$, $\eta^{k}$
\begin{algorithmic}[1]
    \For{$\sigma \in (\sigma_n,\sigma_{n-1},\dots,\sigma_1)$}
    \Comment{$(\sigma_1,\dots,\sigma_n)$ is a total order extending $\Pi$}
        \State $B \gets$ basis of $\big(\im(\eta^{k-1}(\sigma))\big)^{\perp}$ \Comment{see discussion of the algorithm}
        \ForAll{$b\in B$}
            \If{$b$ is linearly independent from the rows of $\eta^k(\sigma)$}
                \State append $\sigma$ to $I^{k+1}$
                \State add $b$ as the next row of $\eta\sb{k}$, labeled by $\sigma$
                
                \Comment{with $0$ entries for places labeled by $\Pi\setminus\St\sigma$}
            \EndIf
        \EndFor
    \EndFor
\end{algorithmic}
\end{algorithm}

\paragraph*{Correctness of Algorithm~\ref{algo:injective_resolution_tail}.}
We claim that starting with minimal injective hull of $F$, iterative application of the algorithm yields the minimal injective resolution of $F$. Each sheaf is injective by definition. We need to show exactness at each point, and minimality.

\begin{proposition}\label{prop:correctness_exactness}
    If $I^k,\eta^{k-1}$ is the input and $I^{k+1},\eta^{k}$ the output of Algorithm~\ref{algo:injective_resolution_tail}, then $I^{k-1}\xrightarrow{\eta^{k-1}} I^k \xrightarrow{\eta^k} I^{k+1}$ is exact in $I^k$.
\end{proposition}
\begin{proof}
We claim that $\im \eta^{k-1}(\sigma)=\ker \eta^k(\sigma)$ for every $\sigma\in\Pi$. Note that once we process an element $\pi$, the submatrix $\eta^k(\pi)$ does not change, because all elements $\tau\in\St\pi$ were already processed before $\pi$.

We analyze step $\sigma$. For all $\pi\in\St\sigma$, $\pi$ was already processed, and we assume that $\im\eta^{k-1}(\pi)=\ker\eta^k(\pi)$.
Since any row labeled by $\St\pi$ has only zero entries in columns labeled by simplices not in $\St\pi$, we start with $\ker\eta^k(\sigma) \supseteq \im\eta^{k-1}(\sigma)$. Adding new rows from $(\im\eta^{k-1}(\sigma))^{\perp}$ preserves this inclusion, and we keep adding new rows until
\begin{equation}\label{eq:algo_correctness_dimensions}
    \rank\eta^k(\sigma) = \dim (\im\eta^{k-1}(\sigma))^{\perp}.
\end{equation}
The dimension $\dim I^k(\sigma)$ is the number of columns of the matrix $\eta^k(\sigma)$, and also the number of rows of $\eta^{k-1}(\sigma)$. Rank-nullity theorem then implies
\begin{equation}\label{eq:algo_correctness_rank-nullity}
    \rank\eta^k(\sigma)+\dim\ker\eta^k(\sigma) = \dim I^k(\sigma) = \dim(\im\eta^{k-1}(\sigma))^{\perp} + \dim\im\eta^{k-1}(\sigma).
\end{equation}
Together, (\ref{eq:algo_correctness_dimensions}) and (\ref{eq:algo_correctness_rank-nullity}) imply that $\dim\ker\eta^k(\sigma) = \dim\im\eta^{k-1}(\sigma)$, and therefore $\ker\eta^k(\sigma)=\im\eta^{k-1}(\sigma)$ at the end of processing $\sigma$, as we claimed.
\end{proof}

To show the minimality, we apply the condition on maximal vectors from Theorem~\ref{thm:main-result}.

\begin{lemma}\label{lem:correctness}
   If $I^{k+1},\eta^k$ is the output of Algorithm~\ref{algo:injective_resolution_tail}, then for all $\sigma\in\Pi$ all maximal vectors in $I^{k+1}(\sigma)$ are in the image of $\eta^k(\sigma)$.
\end{lemma}
\begin{proof}
Maximal vectors over $\sigma$ are exactly the new vectors added at step $\sigma$. That is, if rows $i,\dots,l$ were added to $\eta^k(\sigma)$ at step $\sigma$, the space of maximal vectors is $\Span(e_i,\dots,e_l)$, with $e_j$ being the $j$-th canonical vector. We show that this space is in the image of $\eta^k(\sigma)$.

Let $A_j$ be the matrix consisting of the first $j$ rows of $\eta^k(\sigma)$. As we only add new rows when they are linearly independent from the previous, we have $\ker A_j\subsetneq \ker A_{j-1}$ for every $j\in\{i,\dots,l\}$. That is, there exists $u_j$ such that $A_{j-1} u_j = 0$, and $A_j u_j \neq 0$. This is only possible if $A_j u_j = \lambda e_j$ for some $\lambda\neq 0$, which means that $\eta^k(\sigma) \cdot u_j = A_l u_j$ is a vector with zeros at positions $1,\dots,j-1$, and a non-zero at the $j$-th coordinate. Therefore, $\im\eta^k(\sigma)
\supseteq \Span(\eta^k(\sigma)\cdot u_i, \dots,\eta^k(\sigma)\cdot u_l) 
= \Span(e_i,\dots,e_l)
= M_{I^{k+1}}(\sigma)$.
\end{proof}

\paragraph*{Computing the orthogonal complement.}

In Algorithm~\ref{algo:injective_resolution_tail}, we purposefully leave out any particular way how to compute the basis of $(\im(\eta^{k-1}(\sigma)))^{\perp}$ on line~2, as it is a standard computational problem and it can be implemented in many different ways. One way to compute it is via a standard row reduction algorithm: we start with $U\gets \text{identity matrix}$, $R\gets \eta^{k-1}(\sigma)$, and we reduce rows from top to bottom, reducing each by adding the rows above it to push the left-most non-zero as much to the right as possible. Every row operation performed on $R$ is also performed on $U$, so that $R=U\cdot\eta^{k-1}(\sigma)$. We end up with a lower-triangular matrix $U$ such that all its rows corresponding to the zero rows of $R$ form a basis of $(\im(\eta^{k-1}(\sigma)))^{\perp}$.

An immediate advantage of this approach is that we only ever work with rows of $\eta^{k-1}(\sigma)$. This means we can represent the matrices $\eta^k$ in a row-wise sparse representation, e.g., a list of rows, each represented as a ``\textit{column index} $\rightarrow$ \textit{value}'' dictionary. In this representation, when we go from $\eta^k$ to $\eta^k(\sigma)$, we just choose all the rows labeled by $\St\sigma$---there is no need to crop the rows themselves, as all entries not labeled by $\St\sigma$ are $0$.



\paragraph*{From $F\xrightarrow{\alpha} I^0$ to $I^0\xrightarrow{\eta^0}I^1$.}
In Algorithm~\ref{algo:injective_resolution_tail}, we only use $\eta^{k-1}$ to extract the submatrix $\eta^{k-1}(\sigma)$ and compute the orthogonal complement of its image on line~2. For $k=0$, the inclusion $F\xrightarrow{\alpha} I^0$ plays the role of $\eta^{k-1}$. We can still use the same algorithm, if instead of one matrix $\eta^{k-1}$, we give $\alpha(\sigma)$ for each $\sigma$ as a part of the input. This can be realised for example by calling Algorithm~\ref{algo:injective_hull_alpha} on line 2.

\paragraph*{Constant sheaf.}
For the constant sheaf, the construction of the minimal injective hull $k_{\Pi}\xrightarrow{\alpha} I^0$ is very straightforward. The generators of $I^0$ are the maximal elements of $\Pi$, each exactly once. The map $\alpha$ is given by the diagonal embedding
\begin{align*}
    \alpha(\sigma): k_{\Pi}(\sigma) &\longrightarrow I^0(\sigma) \\
    1 &\mapsto (1,\dots, 1)^T.
\end{align*}
Conveniently, we can represent this particular injection $\alpha$ the same way we represent $\eta^k$. We define $\eta^{-1}$ as a column of ones with rows labeled by the generators of $I^0$, i.e., by the maximal elements of $\Pi$. We can then run Algorithm~\ref{algo:injective_resolution_tail} with no modifications on input $I^0$, $\eta^{-1}$.

\subsection{Examples}

We demonstrate how Algorithm~\ref{algo:injective_resolution_tail} works for constant sheaves with two examples.

\begin{example}\label{ex:4-simplex_with_extra_edges}
Consider the $3$-skeleton of the $4$-simplex, with two extra edges attached to vertex~$1$. We compute the minimal injective resolution of $\Pi\coloneqq\St(1)$. We describe simplices as lists of vertices, and for brevity omit the vertex $1$---e.g., $234=\{1,2,3,4\}$. See Figure~\ref{fig:4-simplex_with_extra_edges_poset}.

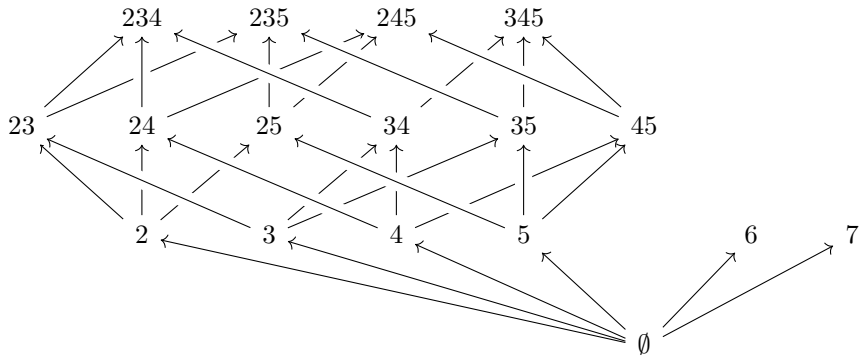
\begin{figure}[htb]
    \centering
\begin{tikzcd}[row sep = large]
    & 234 & 235 & 245 & 345 &  &  & \\
    23 \ar[ru]\ar[rru] & 
    24 \ar[u, crossing over]\ar[rru, crossing over]& 
    25 \ar[u, crossing over]\ar[ru, crossing over]& 
    34 \ar[llu, crossing over]\ar[ru, crossing over]& 
    35 \ar[llu, crossing over]\ar[u, crossing over]& 
    45 \ar[llu, crossing over]\ar[lu, crossing over]&  & \\
    & 
    2 \ar[lu]\ar[u]\ar[ru] & 
    3 \ar[llu, crossing over]\ar[ru]\ar[rru] &
    4 \ar[llu, crossing over]\ar[u, crossing over]\ar[rru] &
    5 \ar[llu, crossing over]\ar[u, crossing over]\ar[ru] & & 6 & 7 \\
    &  &  &  &  & \emptyset \ar[llllu] \ar[lllu] \ar[llu] \ar[lu] \ar[ru] \ar[rru] &  & \\
\end{tikzcd}
    \caption{The poset considered in Example~\ref{ex:4-simplex_with_extra_edges}. We omit vertex $1$ from the labels.}
    \label{fig:4-simplex_with_extra_edges_poset}
\end{figure}

\begin{figure}[htb]
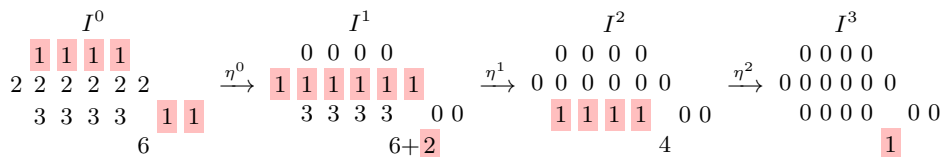

    \centering
{
    \small
    \setlength{\arraycolsep}{1.5pt}
    \newcommand{\frm}[1]{\colorbox{pink}{\makebox[1pt]{#1}}}
    \[
    \begin{NiceArray}{cccccccc}
        &&&\makebox[0pt]{$I^0$}&&&& \\
        & \frm{1} & \frm{1} & \frm{1} & \frm{1} &  &  &   \\
        2 & 2 & 2 & 2 & 2 & 2 &  &   \\
        & 3 & 3 & 3 & 3 &  & \frm{1} & \frm{1}  \\
        &  &  &  &  & 6 &  &   
    \end{NiceArray}
    \ \xrightarrow{\eta^0}\ 
    \begin{NiceArray}{cccccccc}
        &&&\makebox[0pt]{$I^1$}&&&& \\
        & 0 & 0 & 0 & 0 &  &  &   \\
        \frm{1} & \frm{1} & \frm{1} & \frm{1} & \frm{1} & \frm{1} &  &   \\
        & 3 & 3 & 3 & 3 &  & 0 & 0  \\
        &  &  &  &  & \makebox[0pt]{6+\frm{2}} &  &     
    \end{NiceArray}
    \ \xrightarrow{\eta^1}\ 
    \begin{NiceArray}{cccccccc}
        &&&\makebox[0pt]{$I^2$}&&&& \\
        & 0 & 0 & 0 & 0 &  &  &   \\
        0 & 0 & 0 & 0 & 0 & 0 &  &   \\
        & \frm{1} & \frm{1} & \frm{1} & \frm{1} &  & 0 & 0  \\
        &  &  &  &  & 4 &  &      
    \end{NiceArray}
    \ \xrightarrow{\eta^2}\ 
    \begin{NiceArray}{cccccccc}
        &&&\makebox[0pt]{$I^3$}&&&& \\
        & 0 & 0 & 0 & 0 &  &  &   \\
        0 & 0 & 0 & 0 & 0 & 0 &  &   \\
        & 0 & 0 & 0 & 0 &  & 0 & 0  \\
        &  &  &  &  & \frm{1} &  &    
    \end{NiceArray}
    \]
}
    
    \caption{The resolution in Example~\ref{ex:4-simplex_with_extra_edges}. The numbers indicate the dimension at each simplex, positioned as in Figure~\ref{fig:4-simplex_with_extra_edges_poset}. The pink background indicates the generators. For example, in $I^1(\emptyset)$ we have 6 dimensions coming from generators above $\emptyset$, and we have 2 more generators at $\emptyset$.}
    \label{fig:4-simplex_with_extra_edges_resolution}
\end{figure}

The generators of $I^0$ are the maximal simplices $(234, 235, 245, 345, 6, 7)$, and $\eta^{-1}: k_{\Pi}\rightarrow I^0$ is the diagonal embedding for each $\pi\in\Pi$. We construct $I^1$ and $\eta^0$ as in Algorithm~\ref{algo:injective_resolution_tail}, with inputs $I^0,\eta^{-1}$. Initialize $I^1$ and $\eta^0$ empty, and go through the simplices row-by-row left-to-right as they are in Figure~\ref{fig:4-simplex_with_extra_edges_poset}. Starting with $234$, the space $I^0(234)$ is 1-dimensional and equal to $\im\eta^{-1}(234)$, so there is nothing to be added, and $\eta^0(234)=0$. The same happens for all the maximal simplices.

At triangle $23$, we have $I^0(23)=k^2$, since two generators are above $23$. At the moment, $\eta^0(23)$ is empty, so its kernel is $k^2$. We need $\ker\eta^0(23)=\im\eta^{-1}(23)=\Span\{(1,1)\}$. The orthogonal complement of $\im\eta^{-1}(23)$ is generated by the vector $(1,-1)$. We add it as a new row in $\eta^0(23)$. Therefore, we add $23$ to $I^1$, and add a first row to $\eta^0$; see Figure~\ref{fig:4-simplex_with_extra_edges_matrix}. Similarly, we add one row for each other triangle.

Now for the edges. We have $I^0(2)=k^3$, and $\eta^0(2)$ a $3\times 3$ matrix, highlighted as a green solid rectangle in Figure~\ref{fig:4-simplex_with_extra_edges_matrix}. 
We already have $\ker\eta^0(2)=\Span\{(1,1,1)\}=\im\eta^{-1}(2)$, so we do not add any new generators over $2$. The same goes for $\eta^0(3), \eta^0(4), \eta^0(5)$, each of which you can see highlighted in Figure~\ref{fig:4-simplex_with_extra_edges_matrix} with a different color and line style.

Finally, we get to the vertex $\emptyset$, with $\eta^0(\emptyset)$ starting as the part of the matrix in Figure~\ref{fig:4-simplex_with_extra_edges_matrix} above the horizontal line. Its rank is $3$, and its nullity is $3$. We need the kernel to be $1$-dimensional, so we need to add two additional rows from $(\Span\{(1,1,1,1,1,1)\})^\perp$. We also add $\emptyset$ to $I^1$ twice. This completes the construction of $I^1$ and $\eta^0$.

The resolution goes on for two more steps: $I^2$ is generated by $(2,3,4,5)$, $I^3$ by $(\emptyset)$. The matrices $\eta^k$ are in Figure~\ref{fig:4-simplex_with_extra_edges_matrix}, and the whole resolution is schematically shown in Figure~\ref{fig:4-simplex_with_extra_edges_resolution}.

\begin{figure}[htb]
    \centering
    \begin{minipage}{.47\textwidth}
\[
\begin{pNiceArray}[first-row,first-col, margin=.5em]{rrrrrr}
    \eta_0  & 234  & 235  & 245  & 345  & 6  & 7  \\
    23        & 1    & -1   & 0    & 0    & 0  & 0  \\
    24        & 1    & 0    & -1   & 0    & 0  & 0  \\
    25        & 0    & 1    & -1   & 0    & 0  & 0  \\
    34        & 1    & 0    & 0    & -1   & 0  & 0  \\
    35        & 0    & 1    & 0    & -1   & 0  & 0  \\
    45        & 0    & 0    & 1    & -1   & 0  & 0  \\
    \hline
    \emptyset & 0    & 0    & 0    & -1   & 1  & 0  \\
    \emptyset & 0    & 0    & 0    & 0    & -1 & 1
\CodeAfter
    \begin{tikzpicture}
        \node () [fit=(2-1) (2-1), draw=magenta!80!black, rounded corners] {};
        \node () [fit=(2-3) (2-4), draw=magenta!80!black, rounded corners] {};
        \node () [fit=(4-1) (4-1), draw=magenta!80!black, rounded corners] {};
        \node () [fit=(4-3) (4-4), draw=magenta!80!black, rounded corners] {};
        \node () [fit=(6-1) (6-1), draw=magenta!80!black, rounded corners] {};
        \node () [fit=(6-3) (6-4), draw=magenta!80!black, rounded corners] {};
        
        \node () [fit=(3-2) (3-4), draw=red!80!black, dashed, thick] {};
        \node () [fit=(5-2) (6-4), draw=red!80!black, dashed, thick] {};
        
        \node () [fit=(1-1) (1-2), draw=blue!70!black, dotted, thick] {};
        \node () [fit=(1-4) (1-4), draw=blue!70!black, dotted, thick] {};
        \node () [fit=(4-1) (5-2), draw=blue!70!black, dotted, thick] {};
        \node () [fit=(4-4) (5-4), draw=blue!70!black, dotted, thick] {};
        
        \node () [fit=(1-1) (3-3), draw=green!70!black] {};
    \end{tikzpicture}
\end{pNiceArray}
\]
    \end{minipage}\hfill
    \begin{minipage}{.49\textwidth}
\[
\begin{pNiceArray}[first-row,first-col, margin=.5em]{rrrrrrrr}
    \eta_1 & 23 & 24 & 25 & 34 & 35 & 45 & \emptyset & \emptyset \\
    2      & 1  & -1 & 1  & 0  & 0  & 0  & 0  & 0 \\
    3      & 1  & 0  & 0  & -1 & 1  & 0  & 0  & 0 \\
    4      & 0  & 1  & 0  & -1 & 0  & 1  & 0  & 0 \\
    5      & 0  & 0  & 1  & 0  & -1 & 1  & 0  & 0
\CodeAfter
    \begin{tikzpicture}
        \node () [fit=(3-2) (3-2), draw=magenta!80!black, rounded corners] {};
        \node () [fit=(3-4) (3-4), draw=magenta!80!black, rounded corners] {};
        \node () [fit=(3-6) (3-6), draw=magenta!80!black, rounded corners] {};
        
        \node () [fit=(4-3) (4-3), draw=red!80!black, dashed, thick] {};
        \node () [fit=(4-5) (4-6), draw=red!80!black, dashed, thick] {};
        
        \node () [fit=(2-1) (2-1), draw=blue!70!black, dotted, thick] {};
        \node () [fit=(2-4) (2-5), draw=blue!70!black, dotted, thick] {};
        
        \node () [fit=(1-1) (1-3), draw=green!70!black] {};
    \end{tikzpicture}
\end{pNiceArray}
\]
\[
\begin{pNiceArray}[first-row,first-col, margin=.5em]{rrrr}
    \eta_2 & 2 & 3 & 4 & 5 \\
    \emptyset & \hphantom{-}1 & -1 & \hphantom{-}1 & -1
\end{pNiceArray}
\]
    \end{minipage}
    \caption{Matrices $\eta^0$, $\eta^1$, $\eta^2$ in Example~\ref{ex:4-simplex_with_extra_edges}, with highlighted submatrices $\eta^k(2)$ (solid green), $\eta^k(3)$ (dotted blue), $\eta^k(4)$ (rounded corners magenta), $\eta^k(5)$ (dashed red). Recall that $\eta^k(\sigma)=\eta^k[\St\sigma,\St\sigma]$, and note that if $\sigma\not\leq\tau$, then $\eta^k[\sigma,\tau]=0$.}
    \label{fig:4-simplex_with_extra_edges_matrix}
\end{figure}
\end{example}

\begin{example}
    Let $\Sigma\coloneqq\Dnk{3}{2}$ be the $2$-skeleton of a tetrahedron (whose geometric realization is homeomorphic to the sphere). We give the minimal injective resolution of the constant sheaf $k_{\Sigma}$ in Figure~\ref{fig:tetraheron_skeleton}.
    
    \begin{figure}[htb]
        \centering
        \begin{subfigure}{\textwidth}
            \centering
            \begin{tikzcd}[row sep = large]
                & 234 & 134 & 124 & 123 & \\
                34 \ar[ru]\ar[rru] &
                24 \ar[u,crossing over]\ar[rru,crossing over] &
                23 \ar[lu,crossing over]\ar[rru,crossing over] &
                14 \ar[lu,crossing over]\ar[u,crossing over] &
                13 \ar[llu,crossing over]\ar[u,crossing over] &
                12 \ar[llu,crossing over]\ar[lu,crossing over] \\
                & 
                4 \ar[lu]\ar[u]\ar[rru] & 
                3 \ar[llu,crossing over]\ar[u,crossing over]\ar[rru,crossing over] & 
                2 \ar[lu,crossing over]\ar[llu,crossing over]\ar[rru,crossing over] & 
                1 \ar[lu,crossing over]\ar[u,crossing over]\ar[ru,crossing over] &
            \end{tikzcd}
        \end{subfigure}
        
        \vspace{.5cm}
        \begin{subfigure}{\textwidth}
            \centering
            \small
            \setlength{\arraycolsep}{1.5pt}
            \newcommand{\frm}[1]{\colorbox{pink}{\makebox[1pt]{#1}}}
            \[
            \begin{NiceArray}{cccccc}
                &&&\makebox[0pt]{$I^0$}&& \\
                  & \frm{1} & \frm{1} & \frm{1} & \frm{1} & \\
                 2 & 2 & 2 & 2 & 2 & 2 \\
                  & 3 & 3 & 3 & 3 &
            \end{NiceArray}
            \ \xrightarrow{\eta^0}\ 
            \begin{NiceArray}{cccccc}
                &&&\makebox[0pt]{$I^1$}&& \\
                  & 0 & 0 & 0 & 0 & \\
                 \frm{1} & \frm{1} & \frm{1} & \frm{1} & \frm{1} & \frm{1} \\
                  & 3 & 3 & 3 & 3 &
            \end{NiceArray}
            \ \xrightarrow{\eta^1}\ 
            \begin{NiceArray}{cccccc}
                &&&\makebox[0pt]{$I^2$}&& \\
                  & 0 & 0 & 0 & 0 & \\
                 0 & 0 & 0 & 0 & 0 & 0 \\
                  & \frm{1} & \frm{1} & \frm{1} & \frm{1} &
            \end{NiceArray}
            \]
        \end{subfigure}
        \vspace{.5cm}
        \begin{subfigure}{.4\textwidth}
        \[\begin{pNiceArray}[first-row,first-col, margin=.5em]{rrrr}
            \eta^0 & 234 & 134 & 124 & 123 \\
            34 & -1 & 1 & 0 & 0 \\
            24 & -1 & 0 & 1 & 0 \\
            23 & -1 & 0 & 0 & 1 \\
            14 & 0 & -1 & 1 & 0 \\
            13 & 0 & -1 & 0 & 1 \\
            12 & 0 & 0 & -1 & 1 
        \end{pNiceArray}\]
        \end{subfigure}\hspace{1cm}
        \begin{subfigure}{.45\textwidth}
        \[\begin{pNiceArray}[first-row,first-col, margin=.5em]{rrrrrr}
            \eta^1 & 34 & 24 & 23 & 14 & 13 & 12 \\
            4 & 1 & -1 & 0 & 1 & 0 & 0 \\
            3 & 1 & 0 & -1 & 0 & 1 & 0 \\
            2 & 0 & 1 & -1 & 0 & 0 & 1 \\
            1 & 0 & 0 & 0 & 1 & -1 & 1
        \end{pNiceArray}\]
        \end{subfigure}
        
        \caption{Example~\ref{fig:tetraheron_skeleton}: the minimal injective resolution of the constant sheaf on $\Dnk{3}{2}\cong S^2$. First is $\Dnk{3}{2}$ as a poset, second the dimensions in the injective resolution with highlighted generators, as in Figure~\ref{fig:4-simplex_with_extra_edges_resolution}, and third the matrices describing the natural transformations.}
        \label{fig:tetraheron_skeleton}
    \end{figure}
\end{example}

\subsection{Complexity Analysis}

We analyze the complexity of finding the minimal injective resolution of the constant sheaf, $k_{\Pi}$, on a poset $\Pi$, with $n$ elements and height $d$, computed by an iterative application of Algorithm~\ref{algo:injective_resolution_tail}. That is, we start with $k_{\Pi}\xrightarrow{\eta^{-1}}I^0$ the minimal injective hull of the constant sheaf as described above, and then iteratively apply Algorithm~\ref{algo:injective_resolution_tail} until $I^k = 0$.

The body of the outer-most for loop in Algorithm~\ref{algo:injective_resolution_tail} consists of finding a basis of $(\im(\eta^{k-1}(\sigma)))^{\perp}$, and checking for linear independence of rows of $\eta^k(\sigma)$. Both of those operations can be computed in time at most $\mathcal{O}(c^3)$ with $c$ the maximum of the number of rows of $\eta^{k-1}(\sigma)$ and $\eta^{k}(\sigma)$. In our analysis we ignore the complexity of finding $\St\sigma$ to extract the submatrices from $\eta^k$ in the first place, since it is less expensive than $\mathcal{O}(c^3)$ when we estimate $c$ by the size of $\St\sigma$.

By Corollary~\ref{cor:length_of_resolution}, the length of the minimal injective resolution is at most $d+1$. Therefore, we find it in time $\mathcal{O}(d\cdot n\cdot c^3)$, where \[c=\max_{j,\sigma} \sum_{\pi\in\St\sigma} m^j_{k_{\Pi}}(\pi)\]
is the maximal number of generators over any star throughout the resolution.
This analysis is output-sensitive. To give complexity bounds dependent only on the input, we compare $c$ to the maximal size of a star in $\Pi$. How well we can approximate $c$ this way depends on the structure of $\Pi$.

\begin{definition}\label{def:star_complexity}
    For $\sigma\in\Pi$ we define\[
        m^j(\St\sigma)\coloneqq \sum_{\pi\in\St\sigma}m_{k_{\Pi}}^j(\pi)
    \]
    to be the number of generators over $\St\sigma$ in the $j$-th step of the minimal injective resolution of the constant sheaf on $\Pi$. Furthermore, we define the \emph{$j$-th star complexity} of $\sigma$ as \[
        \stcplx^j(\sigma)\coloneqq \frac{m^j(\St\sigma)}{\#\St\sigma}.
    \]
\end{definition}

For general posets, $\stcplx^j(\sigma)$ can be arbitrarily large even when lengths of chains are bounded, because sizes of boundaries and coboundaries can be arbitrarily large.
For simplicial complexes, we give an upper bound on $\stcplx^j(\sigma)$ depending on the dimension.

\begin{proposition}\label{prop:upper_bound_on_SC}
    Let $\Pi$ be a simplicial complex, $\sigma\in\Pi$ and $k\coloneqq \dim\Pi-\dim\sigma$. Then\[
        \stcplx^j(\sigma)\leq\binom{k}{j}.
    \]
    This bound is asymptotically tight. If $\Pi=\Dnk{n}{k}$ is the $k$-skeleton of the $n$-simplex, $v$ is a vertex in $\Dnk{n}{k}$, and $j$ is fixed, then \[
        \stcplx^j(v)\xrightarrow{n\rightarrow\infty}\binom{k}{j}.
   \]
\end{proposition}
\begin{proof}
We prove the upper bound using Theorem~\ref{thm:multiplicities} and bounding dimensions of homology groups by dimensions of chain groups:
\begin{multline*}
    m^j(\St\sigma)
    = \sum_{\tau\in\St\sigma} m^j(\tau)
    = \sum_{\tau\in\St\sigma} \dim H_c^{j+\dim\tau}(\St\tau)
    \leq \sum_{\tau\in\St\sigma} \dim C_c^{j+\dim\tau}(\St\tau)
    \\ = \sum_{\tau\in\St\sigma} \# \left\{ \pi \,\middle|\, \tau<_j\pi \right\}
    = \sum_{\pi\in\St\sigma} \# \left\{ \tau \in \St\sigma \,\middle|\, \tau<_j\pi \right\}
    = \sum_{\pi\in\St\sigma} \#\binom{\pi\setminus\sigma}{j} \leq \#\St\sigma\cdot\binom{k}{j},
\end{multline*}
where $k=\dim\St\sigma-\dim\sigma\leq\dim\Sigma-\dim\sigma$.

Now we analyse $\stcplx^j(\sigma)$ in the $k$-skeleton of the $n$-simplex, $\Dnk{n}{k}$. We use the fact that $\St\sigma$ in $\Dnk{n}{k}$ is combinatorially the same as $\Dnk{n'}{k'}\cup\{\emptyset\}$, with $n'=n-\dim\sigma-1$ and $k'=k-\dim\sigma-1$, using the correspondence $\St\sigma\ni\tau \mapsto \tau\setminus\sigma$. This map induces an isomorphism between the cochain complexes \[
C_c^{\bullet}(\St\sigma) \cong \tilde{C}^{\bullet-\dim\sigma-1}\left(\Dnk{n'}{k'}\right),
\] which, using Theorem~\ref{thm:multiplicities}, implies
\begin{equation*}
    m^j(\sigma) = \dim H_c^{j+\dim\sigma}(\St\sigma) = \dim \tilde{H}^{j-1}\left(\Dnk{n'}{k'}\right).
\end{equation*}
The reduced cohomology $\tilde{H}^i\left(\Dnk{n'}{k'}\right)$ is trivial for all $i\neq k'$, and for $i=k'$, we compute the dimension from the Euler characteristic:
\begin{multline*}
    \dim\tilde{H}^{k'}\left(\Dnk{n'}{k'}\right) = 
    (-1)^{k'} \tilde{\chi} \left(\Dnk{n'}{k'}\right) = 
    (-1)^{k'} \left( 1 + \sum_{i=0}^{k'} \binom{n'+1}{i+1} (-1)^i \right) \\ =
    (-1)^{k'-1} \left( \sum_{i=0}^{k'+1} \binom{n'+1}{i} (-1)^i \right) =
    (-1)^{k'-1} \cdot (-1)^{k'+1} \binom{n'}{k'+1} =
    \binom{n'}{k'+1}.
\end{multline*}
Therefore,
\begin{align*}
   m^j(\sigma) = 
    \begin{cases}
        \binom{n-\dim\sigma-1}{k-\dim\sigma} &\text{ if $j = k'+1 = k-\dim\sigma$,} \\
        0 &\text{ otherwise.}
    \end{cases}
\end{align*}
Finally, we compute $m^j(\St v)$ for a vertex $v$:\[
    m^j(\St v)
    = \sum_{\sigma\in\St v} m^j(\sigma)
    = \sum_{\substack{\sigma\in\St v \\ \dim\sigma=k-j}} \binom{n-k+j-1}{j}
    = \binom{n}{k-j}\cdot\binom{n-k+j-1}{j}
\]

We rearrange this as follows
\begin{align*}
    m^j(\St v)
    &= \frac{n!}{(n-k+j)!\ (k-j)!}\cdot \frac{(n-k+j-1)!}{(n-k-1)!\ j!}
    \\&= \frac{n!}{(n-k)!\ k!}\cdot \frac{n-k}{n-k+j}\cdot  \frac{k!}{(k-j)!\ j!}
    = \binom{n}{k} \binom{k}{j} \frac{n-k}{n-k-1}.
\end{align*}
Now we can easily compare this with $\# \St v = \sum_{i=0}^k \binom{n}{i}$. When we fix $k$ and $j$, we get \[
    \lim_{n\rightarrow\infty} \stcplx^j(v) = \lim_{n\rightarrow\infty} \frac{m^j(\St v)}{\# \St v} = \binom{k}{j}.
\]
\end{proof}

\begin{corollary}\label{cor:complexity}
    For a fixed dimension $d$, the Algorithm~\ref{algo:injective_resolution_tail} computes the minimal injective resolution of the constant sheaf on a $d$-dimensional simplicial complex $\Sigma$ in time $\mathcal{O}(n\cdot s^3)$, where $n$ is the cardinality of $\Sigma$ (as an abstract simplicial complex), and $s$ is the cardinality of the largest star in $\Sigma$.
\end{corollary}

\section{Right Derived Functors}
\label{sec:derived-functors}
As an application of our main results, we define, in terms of injective resolutions, two examples of right derived functors. 

\paragraph*{The right derived pushforward, $R^\bullet f_\ast$.} Let $f:\Sigma\rightarrow \Lambda $ be a continuous (relative to the Alexandrov topology) map of posets. Let $I^\bullet$ be an injective resolution of a sheaf $F$ on $\Sigma$. Define the integers $n^j_F(\pi)$ so that 
\[
I^j = \bigoplus_{\pi\in\Sigma}[\pi]^{n^j_F(\pi)}.
\]
We describe each chain map $\eta^j:I^j\rightarrow I^{j+1}$, as in Section \ref{sec:maps}, by a matrix with columns and rows indexed by the indecomposable summands of $I^j$, and $I^{j+1}$, respectively. 
Let $\eta^j(f^{-1}(\St\lambda))$ be the submatrix of $\eta^j$ consisting of rows and columns corresponding to the indecomposable summands $[\pi]$ with $\pi\in f^{-1}(\St\lambda)$, so that 
\[
\eta^j(f^{-1}(\St\lambda)):\bigoplus_{\pi\in f^{-1}(\St\lambda)}[\pi]^{n^j_F(\pi)}\rightarrow \bigoplus_{\pi\in f^{-1}(\St\lambda)}[\pi]^{n^{j+1}_F(\pi)}.
\] Note that if $\kappa\le\lambda$, then $f^{-1}(\St\lambda)\subset f^{-1}(\St\kappa) $, and the projection
\[\text{proj}:\bigoplus_{\pi\in f^{-1}(\St\kappa)}[\pi]^{n^j_F(\pi)}\rightarrow \bigoplus_{\pi\in f^{-1}(\St\lambda)}[\pi]^{n^j_F(\pi)}\]
induces linear maps
$\ker \eta^j(f^{-1}(\St\kappa))\rightarrow \ker \eta^j(f^{-1}(\St\lambda))$
and \\ 
$\im \eta^{j-1}(f^{-1}(\St\kappa))\rightarrow \im \eta^{j-1}(f^{-1}(\St\lambda))$
(because $\Hom([\pi],[\tau])=0$ if $\tau\not\le\pi$).

\begin{definition}
Define a sheaf $R^jf_\ast F$ on $\Lambda$ by
\[
R^jf_\ast F(\lambda) := \ker \eta^j(f^{-1}(\St\lambda)) / \im\eta^{j-1} (f^{-1}(\St\lambda)), 
\]
with linear maps $R^jf_\ast F(\kappa\le \lambda ) $ induced by the projections described above. 
\end{definition}

\paragraph*{The right derived pushforward with compact support, $R^\bullet f_!$.} Pushforwards with compact support are a critical structure in the machinery of derived categories of sheaves. We would, therefore, like to explicitly describe how to compute $R^\bullet f_! F$ for a given sheaf $F$ and continuous map $f$. However, the topological notion of `compactly supported' does not adapt to the setting of finite posets in a canonical or straightforward way. Subtle topological constraints must be placed on the maps $f$ in order for the discrete calculation to agree with the classical definitions. See \cite[\S 3.3]{Shepard1985} for one approach to establish such topological critiria. In order to keep our methods as transparent and accessible as possible, we will instead describe $R^\bullet f_!$ for a smaller family of functions, which satisfy more familiar topological constraints. We expect that this family of functions is large enough to handle most interesting applications. 

For the remainder of this section, let $\bar{f}:\Sigma\rightarrow\Lambda$ be a simplicial map between finite simplicial complexes. Let $i:U\hookrightarrow \Sigma$ be the inclusion of an open subset $U$ in $\Sigma$. Let $f:U\rightarrow \Lambda$ be the restriction of $\bar{f}$ to $U$. 

\begin{definition}
Given a sheaf $F$ on $U$, define the sheaf $i_!F$ on $\Sigma$ by 
\[
i_!F(\sigma):=\begin{cases}
&F(\sigma)\text{ if }\sigma\in U\\
&0\text{ else,}
\end{cases}
\]
with the linear maps $i_!F(\gamma\le\sigma)= F(\gamma\le\sigma)$ when $(\gamma\le\sigma)\in U$, and 0 else. Finally, define the sheaf $R^jf_! F$ on $\Lambda$ by
\[
R^jf_! F := R^j\bar{f}_\ast (i_! F). 
\]
\end{definition}
\paragraph*{Derived functors and persistent cohomology.} The sheaves $R^jf_\ast F$ and $R^jf_!F$ may be regarded as level-set multi-parameter persistence modules. With this perspective, we can easily compute, from a single injective resolution $I^\bullet$ of $F$, level-set persistence modules associated to any filtration function $f$. Below, we relate the persistence module $R^jf_\ast k_U$ to the singular cohomology of level-sets, and $R^jf_! k_U$ to the compactly supported singular cohomology of level-sets. 
\begin{proposition}\label{prop:persistent-homology}
As sheaves on $\Lambda$, 
\begin{align*}
    R^jf_\ast k_U  &\cong H^j(|f^{-1}(\St\blank)|,k)
\end{align*}
where $H^j(|f^{-1}(\St\blank)|,k)$ is the sheaf defined by associating the simplex $\lambda$ to the singular cohomology of the geometric realization of $f^{-1}(\St\lambda)$ (with linear maps induced by inclusion). 
Moreover, $R^j f_! k_U$ captures the compactly supported singular cohomology of the fibers of $f$:
\begin{align*}
    R^jf_! k_U (\lambda) &\cong H^j_c(|f^{-1}(\St\lambda)|,k).
\end{align*}
\end{proposition}
\begin{proof}
Let $I^\bullet$ be the injective resolution described in Section~\ref{sec:algo_non-inductive}. Let $I^\bullet_{f^{-1}(\St\lambda)}$ be the chain complex of vector spaces consisting of only the linear combinations of generators for indecomposable sheaves $[\pi]\subset I^\bullet$ such that $\pi\in f^{-1}(\St\lambda)$ and chain maps $\eta^\bullet(f^{-1}(\St\lambda))$. Then $I^\bullet_{f^{-1}(\St\lambda)}$ is identical to the simplicial cochain complex of $K(f^{-1}(\St\lambda))$. The cohomology groups of this chain complex are isomorphic to the singular cohomology of the geomtric realization of $f^{-1}(\St\lambda)$: \[R^jf_\ast k_U:=H^j\left(I^\bullet_{f^{-1}(\St\lambda)}\right)\cong H^j(|f^{-1}(\St\lambda)|,k),\] and the linear maps $R^jf_\ast(\kappa\le\lambda)$ are the usual cohomology maps \[H^j(|K(f^{-1}(\St\kappa))|,k)\rightarrow H^j(|K(f^{-1}(\St\lambda))|,k)\] induced by inclusion (cf.\ \cite[Chapter II Proposition 5.11]{Iversen}). A similar argument proves the analogous result for $R^jf_!k_U$. We also note that because $\bar{f}$ is assumed to be a simplicial map between finite simplicial complexes, $\bar{f}$ is proper, and the result follows by applying the proper base change theorem for sheaves (see \cite[Chapter VII Theorem 1.4]{Iversen} or \cite[Proposition 2.6.7]{KashiwaraSchapira1994}).
\end{proof}

\section{Discussion}\label{sec:discussion}
An injective resolution represents a given sheaf with an exact sequence of injective sheaves. The homological properties of the given sheaf can then be deduced from its injective resolution (which has many more theoretically and practically desirable properties). The results of this paper address fundamental aspects of computing injective resolutions. First, we prove the existence and uniqueness of a minimal injective resolution, and provide several of its defining characteristics (Theorem \ref{thm:main-result} and Corollary \ref{cor:existence}). We give a topological interpretation of the multiplicities of indecomposable injective sheaves in the minimal injective resolution of the constant sheaf over a simplicial complex (Theorem \ref{thm:multiplicities}). We introduce two new methods for constructing injective resolutions. The first defines the $k$-th term of the resolution without referencing earlier terms (Section~\ref{sec:algo_non-inductive}). The second is an inductive algorithm which computes the minimal injective resolution of a given sheaf (Section~\ref{sec:algo_minimal}). Finally, we give asymptotically tight bounds on the complexity of computing the minimal injective resolution of the constant sheaf on a simplicial complex using Algorithm~\ref{algo:injective_resolution_tail} (Proposition \ref{prop:upper_bound_on_SC} and Corollary \ref{cor:complexity}).

There are many directions in which to extend this work, and several interesting questions which arise from studying derived categories of sheaves from the perspective of computational topology. 
To make full use of the derived category machinery in computational topology, it is necessary to develop algorithms for computing each of Grothendieck's six functors on derived categories: $f_\ast$, $f^\ast$, $f_!$, $f^!$, $\Hom$, and $\otimes$. To this end, it will be useful to extend the results of this paper to injective resolutions of complexes of sheaves: to each complex $F^\bullet$, compute a quasi-isomorphic complex of injective sheaves $I^\bullet$. We plan to pursue this in future work. 

Theorem~\ref{thm:multiplicities} also suggests an interesting connection between the minimal injective resolution, the $\bar{p}$-canonical stratifications of Goresky--MacPherson~\cite{GoreskyMacPherson}, and the cohomological stratification of Nanda~\cite{Nanda}. Briefly, the canonical $\bar{p}$-stratification is constructed by inductively identifying subsets of $\Sigma$ for which the dualizing complex $\omega^\bullet_\Sigma$ is cohomologically locally constant. Similarly, the cohomological stratification of \cite{Nanda} is constructed by inductively identifying subsets of $\Sigma$ for which the cosheaf $\sigma\mapsto H_c^j(\St\sigma,k)$ is locally constant. The (co)homology groups $H^{-j}(\omega^\bullet_\Sigma(\sigma))\cong H_j(\Sigma, \Sigma - \St\sigma)$ and $H_c^j(\St\sigma,k)$, are closely related, by Theorem \ref{thm:multiplicities}, to the multiplicities of indecomposable injective sheaves in the minimal injective resolution of the constant sheaf. However, to compute the canonical $\bar{p}$-stratification or cohomological stratification, it is necessary to investigate the linear maps induced by these (co)sheaves. It is not currently clear to us how we can recover the linear maps between cohomology groups: $H_c^j(\St\sigma,k)\rightarrow H_c^j(\St\tau,k)$, from the minimal injective resolution of $k_\Sigma$, without taking a barycentric subdivision of the simplicial complex. However, we conjecture that the invertibility of such linear maps can be deduced from the minimal injective resolution of $k_\Sigma$. If this is true, then it would be possible to efficiently compute canonical $\bar{p}$-stratifications and cohomological stratifications directly from the minimal injective resolution of $k_\Sigma$.



\end{document}